%% file: Arxiv_submission.tex
\documentclass[11pt]{article}

\usepackage{epsfig,epsf,fancybox}
\usepackage{amsmath}
\usepackage{mathrsfs}
\usepackage{amssymb}
\usepackage{graphicx}
\usepackage{color}
\usepackage{multirow}
\usepackage{paralist}
\usepackage{verbatim}
\usepackage{galois}
\usepackage{algorithm}
\usepackage{algorithmic}
\usepackage{boxedminipage}
\usepackage{booktabs}
\usepackage{accents}
\usepackage{stmaryrd}

\usepackage{subfig}

\usepackage{natbib}

\usepackage{url}% for url's in bib
\usepackage[colorlinks,linkcolor=magenta,citecolor=blue, pagebackref=true,backref=true]{hyperref}
\renewcommand*{\backrefalt}[4]{%
    \ifcase #1 \footnotesize{(Not cited.)}%
    \or        \footnotesize{(Cited on page~#2.)}%
    \else      \footnotesize{(Cited on pages~#2.)}%
    \fi}

\newtheorem{theorem}{Theorem}[section]

\newtheorem{lemma}[theorem]{Lemma}

\newtheorem{definition}[theorem]{Definition}

\newtheorem{remark}[theorem]{Remark}

\numberwithin{equation}{section}

\newcommand{\sign}{\textnormal{sign}}

\newcommand{\op}{\textnormal{op}}

\newcommand{\argmin}{\mathop{\rm argmin}}

\newcommand{\ECal}{\mathcal{E}}

\newcommand{\HCal}{\mathcal{H}}

\newcommand{\XCal}{\mathcal{X}}

\newcommand{\br}{\mathbb{R}}

\newcommand{\ba}{\begin{array}}
\newcommand{\ea}{\end{array}}

\begin{document}

%%%%%%% TITLE PAGE %%%%%%%%%%%%%%%%%%%%%%%%%%%%%%%%%%%%%%%%%%%%%%%%%%%

\begin{center}

{\bf{\LARGE{A Continuous-Time Perspective on Global Acceleration \\ [.2cm] for Monotone Equation Problems}}}

\vspace*{.2in}
{\large{ \begin{tabular}{c}
Tianyi Lin$^\ddagger$ \and Michael I. Jordan$^{\diamond, \dagger}$ \\
\end{tabular}
}}

\vspace*{.2in}

\begin{tabular}{c}
Department of Industrial Engineering and Operations Research, Columbia University$^\ddagger$ \\
Department of Electrical Engineering and Computer Sciences$^\diamond$ \\
Department of Statistics$^\dagger$ \\
University of California, Berkeley
\end{tabular}

\vspace*{.2in}

\today

\vspace*{.2in}

\begin{abstract}
We propose a new framework to design and analyze accelerated methods that solve general monotone equation (ME) problems in the form of $F(x) = 0$ where $F$ is a continuous and \textit{monotone} operator. Traditional approaches include generalized steepest descent methods~\citep{Hammond-1987-Generalized} and inexact Newton-type methods~\citep{Solodov-1999-Globally}. If $F$ is \textit{uniformly monotone} and \textit{twice differentiable}, the former methods are shown to achieve local linear convergence while the latter methods attain a guarantee of local superlinear convergence. The latter methods are also globally convergent thanks to line search and hyperplane projection. However, a global convergence rate is unknown for these methods. Since the ME problems are equivalent to the unconstrained variational inequality (VI) problems,  VI methods can be applied to yield a global convergence rate that is expressed in terms of the residue function $\|F(x)\|$. Indeed, the optimal rate among first-order methods is $O(1/k)$ if $F$ is Lipschitz continuous and was achieved by an anchored extragradient method~\citep{Yoon-2021-Accelerated}. However, these existing results are restricted to first-order methods and ME problems with a Lipschitz continuous operator. It has not been clear how to obtain global acceleration using high-order Lipschitz continuity of $F$. 

We take a \textit{continuous-time} perspective in which accelerated methods are viewed as the discretization of dynamical systems. Our contribution is to propose \textit{accelerated rescaled gradient systems} and prove that they are equivalent to closed-loop control systems~\citep{Lin-2023-Monotone}. Based on this connection, we establish the desired properties of solution trajectories of our new systems. Moreover, we provide a unified algorithmic framework obtained from discretization of continuous-time systems, which together with two different approximation subroutines yields both the existing high-order methods~\citep{Lin-2022-Perseus} and new first-order methods. This highlights the advantage of our new systems over the closed-loop control systems since the resulting methods do not require line search at each iteration. We prove that the $p^\textnormal{th}$-order method achieves a global rate of $O(k^{-p/2})$ in terms of $\|F(x)\|$ if $F$ is \textit{$p^\textnormal{th}$-order Lipschitz continuous} and the first-order method can achieve the same rate if $F$ is \textit{$p^\textnormal{th}$-order strongly Lipschitz continuous}. If $F$ is further strongly monotone, the restarted versions of these methods achieve local convergence with order $p$ when $p \geq 2$. Our discrete-time analysis is largely motivated by the continuous-time analysis and demonstrates the fundamental role that rescaled gradients play in global acceleration for solving ME problems.

\vspace{0.5cm}

\end{abstract}

\vspace{0.3cm}

{\large \sf This paper is dedicated to the special issue in memory of Professor Hedy Attouch}

\end{center}

\newpage

\input{sec/intro}

\input{sec/system}

\input{sec/algorithm}

\input{sec/analysis}

\input{sec/exp}

\input{sec/conclu}

\section*{Acknowledgments}
This work was supported in part by the Mathematical Data Science program of the Office of Naval Research under grant number N00014-21-1-2840, and by the Vannevar Bush Faculty Fellowship program
under grant number N00014-21-1-2941, and by the European Research Council Synergy Program.

\bibliographystyle{plainnat}
\bibliography{ref}

\end{document}

%% file: sec/intro.tex
%!TEX root = paper.tex
\section{Introduction}
Monotone equations (MEs) capture a wide range of problems that include convex optimization problems, convex-concave saddle-point problems and computational models of equilibria in game-theoretic settings. Let $\HCal$ be a real Hilbert space and we define $F(x): \HCal \mapsto \HCal$ as a continuous operator. In the context of ME problem, we assume that $F$ is \textit{monotone}, i.e., $\langle F(x) - F(y), x- y \rangle \geq 0$ for any $x, y \in \HCal$, and want to find a point $x^\star \in \HCal$ such that
\begin{equation}\label{prob:main}
F(x^\star) = 0.  
\end{equation}
Note that the ME problem can be equivalently formulated as unconstrained variational inequality (VI) problems corresponding to $F$ and $\HCal$~\citep{Facchinei-2007-Finite}. In particular, we have that $x^\star \in \HCal$ is the solution to Eq.~\eqref{prob:main} if and only if we have
\begin{equation*}
\langle F(x^\star), x- x^\star\rangle \geq 0, \quad \textnormal{for all } x \in \HCal.
\end{equation*}
There are many application problems that can be directly cast into the ME framework, including problems in economics and game theory~\citep{Morgenstern-1953-Theory, Fudenberg-1998-Theory}, interval arithmetic~\citep{Moore-1979-Methods}, kinematics~\citep{Morgan-1987-Computing}, chemical engineering~\citep{Meintjes-1990-Chemical}, neurophysiology~\citep{Verschelde-1994-Homotopies} and other engineering problems~\citep{Morgan-2009-Solving}. For an overview of recent progress on solving nonlinear equations and many more applications, we refer to~\citet{Heath-2018-Scientific}. This line of research has recently been brought into contact with machine learning (ML), in the guise of optimality conditions for saddle-point problems, with concrete applications including robust prediction and regression~\citep{Xu-2009-Robustness, Esfahani-2018-Data}, adversarial learning~\citep{Goodfellow-2020-Generative}, online learning in games~\citep{Cesa-2006-Prediction} and distributed computing~\citep{Shamma-2008-Cooperative}. Emerging ML applications often involve multi-agent systems, multi-way markets, or social context, and this is driving increasing interest in equilibrium formulations that can be solved as ME problems.

The classical iterative method for solving the ME problems is so-called generalized steepest descent method~\citep{Bertsekas-1982-Constrained, Hammond-1987-Generalized}. Although this algorithm is not globally convergent for the general case where the Jacobian of $F$ is asymmetric,~\citet{Hammond-1987-Generalized} established a local linear convergence guarantee when $F$ is uniformly monotone and twice differentiable. The global convergence was proved when $F(x) = Mx - b$ and both $M$ and $M^2$ are positive definite. The subsequent work of~\citet{Magnanti-1997-Averaging} combined the generalized steepest descent method with averaging and proved the global convergence results under weaker assumption that $M^2$ is semidefinite. An alternative state-of-the-art method for solving the ME problem is the Newton method which enjoys a local superlinear convergence guarantee under certain regularity conditions~\citep{Kelley-1995-Iterative, Kelley-2003-Solving}. Higher-order methods were also proposed that improve upon Newton method---they are guaranteed to achieve local superquadratic convergence under the high-order Lipschitz continuity of $F$~\citep{Homeier-2004-Modified, Darvishi-2007-Third, Cordero-2007-Variants, Mcdougall-2014-Simple}. As for the global convergence, while it is indeed possible to give an asymptotic guarantee using line search and hyperplane projection~\citep{Solodov-1999-Globally}, the \textit{global convergence rate} for these methods remains unknown. 

Since the ME problems are equivalent to unconstrained variational inequality (VI) problems, the VI methods can be applied with a global convergence rate guarantee. Examples include the extragradient (EG) method~\citep{Korpelevich-1976-Extragradient, Nemirovski-2004-Prox}, optimistic gradient method~\citep{Popov-1980-Modification, Mokhtari-2020-Convergence}, dual extrapolation method~\citep{Nesterov-2007-Dual} and other methods~\citep{Solodov-1999-Hybrid, Monteiro-2010-Complexity, Monteiro-2011-Complexity}. All of these methods are first order and achieve the same rate of $O(1/k)$ in terms of \textit{a restricted gap function} if $F$ is Lipschitz continuous. In terms of a residue function $\|F(x)\|$, the rate of $O(1/\sqrt{k})$ has been derived for the EG method~\citep{Solodov-1999-Hybrid, Gorbunov-2022-Extragradient} and optimistic gradient method~\citep{Ryu-2019-ODE, Kotsalis-2022-Simple}. It was later improved to an optimal rate of $O(1/k)$ by an anchored EG method~\citep{Yoon-2021-Accelerated}. However, these existing results were restricted to first-order methods and the ME problem with a Lipschitz continuous operator. It remains unclear how to obtain global acceleration using high-order Lipschitz continuity of $F$. 

As noted by~\citet{Monteiro-2012-Iteration}, a global acceleration is obtained by second-order methods if $F$ is second-order Lipschitz continuous. More specifically, the $p^\textnormal{th}$-order methods can be used to solve the ME problems when $F$ is $p^\textnormal{th}$-order Lipschitz continuous~\citep{Bullins-2022-Higher, Jiang-2022-Generalized, Adil-2022-Optimal, Lin-2022-Perseus, Lin-2023-Monotone}. However, the convergence analysis depends on a heuristic Lyapunov function and the global rate is \textit{only} obtained in terms of a restricted gap function which is arguably \textit{less natural} than a residue function for the ME problems.  

This paper is motivated by the \textit{continuous-time} perspective of interpreting accelerated methods as the discretization of dynamical systems. In particular, we propose novel \textit{accelerated rescaled gradient systems} for characterizing global acceleration for ME problems in Eq.~\eqref{prob:main}. Formally, our new systems can be written as follows:
\begin{equation}\label{sys:main}
\begin{array}{lll}
\dot{s}(t) = - \tfrac{F(x(t))}{\|F(x(t))\|^{1-1/p}}, & v(t) = x_0 + s(t), & x(t) - v(t) + \tfrac{F(x(t))}{\|F(x(t))\|^{1-1/p}} = 0, 
\end{array} 
\end{equation}
where $\theta > 0$ and $p \in \{1, 2, \ldots\}$. The initial condition is $s(0) = 0$ and $x(0) = x_0 \in \{x \in \HCal \mid F(x) \neq 0\}$. Note that this is not a restrictive condition since $F(x) = 0$ implies that $x \in \HCal$ is a solution to Eq.~\eqref{prob:main}. Throughout the paper, unless otherwise indicated, we assume that 
\begin{quote}
\centering
\textit{$F: \HCal \rightarrow \HCal$ is continuous and monotone and Eq.~\eqref{prob:main} has at least one solution.}
\end{quote}
\paragraph{Our Contributions.} We summarize the main contributions of this paper as follows: 
\begin{enumerate}
\item We show that the accelerated rescaled gradient systems in Eq.~\eqref{sys:main} are equivalent to closed-loop control systems~\citep{Lin-2023-Monotone}. This result is surprising since our systems are open loop without any explicit control law involved. Based on this fact, we establish the desired properties of solution trajectories, including global existence and uniqueness (see Theorem~\ref{Thm:Global-Existence-Uniquess}) and asymptotic weak and strong asymptotic convergence (see Theorem~\ref{Thm:Rate-Asymptotic}). Our results also include a convergence rate estimate in terms of a residue function (see Theorem~\ref{Thm:Rate-Nonasymptotic}). 

\item We provide a unified algorithmic frameworks obtained from discretization of the system in Eq.~\eqref{sys:main}. When combined with two different approximate subroutines it yields the classical $p^\textnormal{th}$-order VI methods specialized to solve ME problems~\citep{Lin-2022-Perseus} and new first-order methods. This highlights the advantage of our new systems over the closed-loop control systems since the resulting methods do not require line search at each iteration. We prove that the $p^\textnormal{th}$-order method achieves a global rate of $O(1/k^{p/2})$ in terms of $\|F(x)\|$ if $F$ is \textit{$p^\textnormal{th}$-order Lipschitz continuous} (see Theorem~\ref{Thm:Global-Rate-High}) and the first-order method achieves the same global rate if $F$ is \textit{$p^\textnormal{th}$-order strongly Lipschitz continuous} (see Theorem~\ref{Thm:Global-Rate-First}). If $F$ is further strongly monotone, the restarted versions of these methods can achieve a local convergence with order $p$ when $p \geq 2$ (see Theorem~\ref{Thm:Local-Rate-Estimate}). All of these results are new to the best of our knowledge. 
\end{enumerate}
Our discrete-time analysis is largely motivated by the continuous-time analysis and demonstrates the fundamental role that rescaled gradients play in global acceleration for solving the ME problem. 

\paragraph{Additional Related Works.} The interplay between continuous-time dynamics and discrete-time methods has been a milestone in optimization theory. A flurry of recent work in convex optimization has focused on uncovering a general principle underlying Nesterov's accelerated gradient method (NAG) with a focus on interpreting the acceleration phenomenon using a temporal discretization of a continuous-time dynamical system with damping term~\citep{Su-2016-Differential, Wibisono-2016-Variational,Attouch-2016-Fast, Attouch-2018-Fast, Diakonikolas-2019-Approximate, Wilson-2021-Lyapunov, Muehlebach-2019-Dynamical,Muehlebach-2021-Optimization, Muehlebach-2022-Constraints,Shi-2022-Understanding,Lin-2022-Control}. Among the aforementioned works,~\citet{Wilson-2021-Lyapunov} found an unified time-dependent Lyapunov function and showed that their Lyapunov analysis is equivalent to Nesterov’s estimate sequence technique in a number of cases, including quasi-monotone subgradient, accelerated gradient descent and conditional gradient.

Extending the continuous-time dynamics and Lyapunov function analysis from convex optimization to monotone variational inequality and inclusion problems has been an active research area during the last two decades~\citep{Attouch-2011-Continuous, Mainge-2013-First, Abbas-2014-Newton, Attouch-2013-Global, Attouch-2016-Dynamic, Attouch-2020-Newton, Attouch-2021-Continuous, Lin-2023-Monotone}. In particular,~\citet{Attouch-2016-Dynamic} has proposed a second-order method for solving monotone inclusion problems but obtained a convergence rate estimate only in an optimization setting. \citet{Lin-2023-Monotone} proposed a closed-loop damping approach that generalizes~\citet{Attouch-2016-Dynamic} and gave high-order methods with convergence rate estimates for monotone inclusion problems. In this context, the Lyapunov function has an intuitive interpretation of being the distance between the solution trajectory and an optimal solution. 

Another line of works leveraged the asymptotic pseudo-trajectory (APT) theory~\citep{Benaim-1996-Asymptotic} to push the trajectory convergence in the continuous-time dynamics to the last-iterate convergence for stochastic discrete-time methods~\citep{Benaim-2005-Stochastic, Mertikopoulos-2018-Stochastic, Zhou-2017-Mirror, Zhou-2020-Convergence}. For example,~\citet{Zhou-2020-Convergence} examined the asymptotic convergence of stochastic mirror descent (SMD) in a class of optimization problems that are not necessarily convex and proved that SMD reaches a minimum point in a finite number of steps (a.s) in problems with sharp minima. In contrast, our work presented a unified framework to develop a class of deterministic first-order and high-order methods for monotone equation problems and derive their convergence rate in terms of residue norm.  

The closest work to ours is~\citet{Wilson-2019-Accelerating}, who derived novel first-order optimization methods via appeal to the discretization of \textit{rescaled gradient system} and proved that their accelerated methods can achieve the same global rate as accelerated high-order methods under a strong smoothness condition. While our systems in Eq.~\eqref{sys:main} are built upon the idea of scaling and our strong Lipschitzness condition is inspired by their condition,~\citet{Wilson-2019-Accelerating} considered a fundamentally different clas of dynamical systems without any continuous-time analysis and the derivations of their accelerated methods do not flow from a single underlying principle but tend to involve case-specific algebra.

\paragraph{Organization.} The remainder of the current paper is organized as follows. In Section~\ref{sec:system}, we establish  global existence and uniqueness results and analyze the convergence properties of solution trajectories. In Section~\ref{sec:algorithm}, we propose two algorithmic frameworks via appeal to the discretization of the accelerated rescaled gradient systems in Eq.~\eqref{sys:main}, one of which when combined with an approximate tensor subroutine recovers the existing $p^\textnormal{th}$-order VI methods and the other of which yields a new suite of simple first-order methods. In Section~\ref{sec:discrete}, we provide a global convergence rate estimate in terms of a residue function. In Section~\ref{sec:exp}, we conduct experiments on synthetic data to validate the efficiency of our methods. In Section~\ref{sec:conclusions}, we conclude the paper with a discussion on future research directions.

\paragraph{Notation.} We use lower-case letters such as $x$ to denote vectors, and upper-case letters such as $X$ to denote tensors.  We let $\HCal$ be a real Hilbert space which is endowed with an inner product $\langle \cdot, \cdot\rangle$. For $x \in \HCal$, we let $\|x\|$ denote the norm induced by $\langle \cdot, \cdot\rangle$. If $\HCal = \br^d$ is a real Euclidean space,  $\|x\|$ reduces to the $\ell_2$-norm of $x$.  Then, for $X \in \br^{d_1 \times d_2 \times \ldots \times d_p}$, we define 
\begin{equation*}
X[z^1, \cdots, z^p] = \sum_{1 \leq i_j \leq d_j, 1 \leq j \leq p} (X_{i_1, \cdots, i_p})z_{i_1}^1 \cdots z_{i_p}^p, 
\end{equation*}
and let $\|X\|_\op = \max_{\|z^i\|=1, 1 \leq j \leq p} X[z^1, \cdots, z^p]$.  Fixing $k \geq 1$, we say $F: \br^d \rightarrow \br^d$ is $k^\textnormal{th}$-order Lipschitz continuous if there exists a positive constant $L > 0$ such that $\|D^{(k-1)} F(x') - D^{(k-1)} F(x)\| \leq L\|x' - x\|$ for all $x, x' \in \br^d$. Here, $D^{(j)} F(x)$ stands for the $j^\textnormal{th}$-order derivative of $F$ at $x \in \br^d$ and $D^{(0)} F = F$; indeed, for $\{z_1, z_2, \ldots, z_j\} \subseteq \br^d$, we have
\begin{equation*}
D^{(j-1)} F(x)[z^1, \cdots, z^j] = \sum_{1 \leq i_1, \ldots, i_j \leq d} \left[\tfrac{\partial F_{i_1}}{\partial x_{i_2} \cdots \partial x_{i_j}}(x)\right] z_{i_1}^1 \cdots z_{i_j}^j. 
\end{equation*}
Lastly, the notation $a = O(b(k))$ stands for an upper bound $a \leq C \cdot b(k)$, where $C > 0$ is independent of the iteration count $k \in \{1, 2, \ldots\}$ and $a = \tilde{O}(b(k))$ indicates the same inequality where $C > 0$ might depend on the logarithmic factors of $k$.

%% file: sec/system.tex
%!TEX root = paper.tex
\section{Continuous-Time Accelerated Rescaled Gradient Systems}\label{sec:system}
We begin by showing that the accelerated rescaled gradient systems in Eq.~\eqref{sys:main} are equivalent to the closed-loop control systems proposed by~\citet{Lin-2023-Monotone}.  Based on this fact, we use dynamical systems concepts to establish the existence and uniqueness of a global solution of Eq.~\eqref{sys:main} and to prove asymptotic weak and strong convergence of solution trajectories. We also propose a simple yet novel Lyapunov function which we employ to derive nonasymptotic global convergence rates in terms of a residue function. 

\subsection{Reformulation via closed-loop control}
We rewrite the accelerated rescaled gradient system in Eq.~\eqref{sys:main} as follows:
\begin{equation*}
\begin{array}{lll}
\dot{s}(t) = - \tfrac{F(x(t))}{\|F(x(t))\|^{1-1/p}}, & v(t) = x_0 + s(t), & x(t) - v(t) + \tfrac{F(x(t))}{\|F(x(t))\|^{1-1/p}} = 0. 
\end{array} 
\end{equation*}
Introducing the function $\lambda(t) = \frac{1}{\|F(x(t))\|^{1-1/p}}$, we have
\begin{equation}\label{Prelim:Equal-first}
\begin{array}{lll}
\dot{s}(t) = - \lambda(t) F(x(t)), & v(t) = x_0 + s(t), & x(t) - v(t) + \lambda(t) F(x(t)) = 0. 
\end{array}
\end{equation}
Since $s(t) = v(t) - x_0$ and $s(0) = 0$, we have $v(0) = x_0$ and 
\begin{equation*}
\dot{v}(t) = \dot{s}(t) = - \lambda(t) F(x(t)) = x(t) - v(t). 
\end{equation*}
In addition, we have $x(t) = (I + \lambda(t)F)^{-1} v(t)$.  Thus, we can rewrite Eq.~\eqref{Prelim:Equal-first} as 
\begin{equation*}
\left\{\begin{array}{ll}
& \dot{v}(t) = (I + \lambda(t)F)^{-1} v(t) - v(t), \\
& s(t) = v(t) - x_0, \\
& x(t) = (I + \lambda(t)F)^{-1} v(t). 
\end{array}\right. 
\end{equation*}
This implies that $(x, s)$ depend on the variables $(v, \lambda)$ explicitly and can be eliminated from the system. By the definition of $\lambda(\cdot)$, we have 
\begin{equation*}
1 = (\lambda(t))^p\|F(x(t))\|^{p-1} = \lambda(t)\|\lambda(t)F(x(t))\|^{p-1} = \lambda(t)\|\dot{v}(t)\|^{p-1}. 
\end{equation*}
Summarizing, the accelerated rescaled gradient systems in Eq.~\eqref{sys:main} can be equivalently reformulated as closed-loop control systems~\citep[cf.][]{Attouch-2016-Dynamic, Lin-2023-Monotone} as follows:
\begin{equation}\label{Prelim:Equal-second}
\left\{\begin{array}{ll}
& \dot{v}(t) = (I + \lambda(t)F)^{-1} v(t) - v(t), \\ 
& \lambda(t)\|\dot{v}(t)\|^{p-1} = 1, \\
& v(0) = x_0 \in \Omega = \{x \in \HCal \mid F(x) \neq 0\}.
\end{array} \right.
\end{equation}
\begin{remark}
The above relationship between the accelerated rescaled gradient systems in Eq.~\eqref{sys:main} and the closed-loop systems in Eq.~\eqref{Prelim:Equal-second} pave the way for proving global existence and uniqueness as well as asymptotic weak and strong convergence via appeal to the existing results for closed-loop control systems studied by~\citet[Theorem~2.7, 3.1 and 3.6]{Lin-2023-Monotone}. Such a closed-loop formulation also provides a pathway to establishing global convergence result results that are distinct from those employed in other rescaled gradient systems associated with convex optimization methods~\citep{Cortes-2006-Finite, Wibisono-2016-Variational, Wilson-2019-Accelerating, Romero-2020-Finite}. For further discussion of dynamical systems approaches to optimization, we refer to~\citet{Romero-2020-Finite}.
\end{remark}

\subsection{Main results}
We state our main theorem on the existence and uniqueness of a global solution. 
\begin{theorem}\label{Thm:Global-Existence-Uniquess}
The accelerated rescaled gradient system with a fixed value of $p \geq 1$ has a unique global solution, $(x, v, s): [0, +\infty) \mapsto \HCal \times \HCal \times \HCal$. In addition, we have that $x(\cdot)$ is continuous, $v(\cdot)$ and $s(\cdot)$ are continuously differentiable, and $\|F(x(t))\|^{1/p-1}$ is nondecreasing as a function of $t$. If $p \geq 2$, we have 
\begin{equation*}
\|F(x(t))\| \geq \|F(x_0)\| e^{-\frac{pt}{p-1}}, \quad \textnormal{for all } t \geq 0. 
\end{equation*}
\end{theorem}
\begin{proof}
By~\citet[Theorem~2.7 and Lemma~2.11]{Lin-2023-Monotone}, the closed-loop control system in Eq.~\eqref{Prelim:Equal-second} has a unique global solution, $(v, \lambda): [0, +\infty) \mapsto \HCal \times (0, +\infty)$.  In addition, $v(\cdot)$ is continuously differentiable and $\lambda(\cdot)$ is locally Lipschitz continuous and nondecreasing. If $p \geq 2$, we have 
\begin{equation}\label{GEU:inequality-first}
\|v(t) - (I + \lambda(t)F)^{-1}v(t)\| \geq \|v(0) - (I + \lambda(0) F)^{-1}v(0)\|e^{-\frac{t}{p-1}}, \quad \textnormal{for all } t \geq 0. 
\end{equation}
Recall that Eq.~\eqref{sys:main} can be reformulated as Eq.~\eqref{Prelim:Equal-second} with the following transformations: 
\begin{equation*}
s(t) = v(t) - x_0, \quad x(t) = (I + \lambda(t)F)^{-1} v(t), \quad \lambda(t) = \tfrac{1}{\|F(x(t))\|^{1-1/p}}. 
\end{equation*}
Since $v(\cdot)$ is continuously differentiable, we have $s(\cdot)$ is continuously differentiable. Since $\lambda(\cdot)$ is locally Lipschitz continuous and nondecreasing, we obtain from the definition of $\lambda(t)$ that $\|F(x(t))\|^{1/p-1}$ is nondecreasing as a function of $t$. As for the relationship between $x(\cdot)$ and $v(\cdot)$, we have
\begin{equation*}
v(t) = x(t) + \lambda(t) F(x(t)) = x(t) + \tfrac{F(x(t))}{\|F(x(t))\|^{1-1/p}}, 
\end{equation*}
which implies that $\|x(t) - v(t)\| = \|F(x(t))\|^{1/p}$ and further we have
\begin{equation}\label{GEU:inequality-second}
F(x(t)) + \|x(t) - v(t)\|^{p-1}(x(t) - v(t)) = 0, \quad \textnormal{for all } t \geq 0. 
\end{equation}
We claim that $x(\cdot)$ is uniquely determined by $v(\cdot)$ and that $x(\cdot)$ is continuous. Indeed, let $x_1(t)$ and $x_2(t)$ both satisfy Eq.~\eqref{GEU:inequality-second} and let $x_1(t_0) \neq x_2(t_0)$ for some $t_0 > 0$.  For simplicity, we assume that $x_1 = x_1(t_0)$, $x_2 = x_2(t_0)$, $v = v(t_0)$ and $\tilde{F}(\cdot) = F(\cdot) + \|\cdot - v\|^{p-1}(\cdot-v)$. Then, we have
\begin{equation*}
\tilde{F}(x_1) = \tilde{F}(x_2) = 0. 
\end{equation*}
Define $f(\cdot) = \frac{1}{p+1}\|\cdot - v\|^{p+1}$, we have $\tilde{F}(\cdot) = F(\cdot) + \nabla f(\cdot)$. Since $f$ is strictly convex and $F$ is monotone, we have that $\tilde{F}$ is strictly monotone. Putting these pieces together yields that $x_1 = x_2$. By the definition of $x_1$ and $x_2$, this contradicts $x_1(t_0) \neq x_2(t_0)$. As such, $x(\cdot)$ is uniquely determined by $v(\cdot)$. 

We next prove the continuity of $x(\cdot)$ by contradiction. Assume that there exists $t_0 > 0$ such that $x(t_n) \rightarrow \tilde{x} \neq x(t_0)$ for some sequence $t_n \rightarrow t_0$.  Note that Eq.~\eqref{GEU:inequality-second} implies that 
\begin{equation*}
F(x(t_n)) + \|x(t_n) - v(t_n)\|^{p-1}(x(t_n) - v(t_n)) = 0. 
\end{equation*}
Taking the limit of this equation with respect to the sequence $t_n$ and using the continuity of $F$ and $v(\cdot)$, we have
\begin{equation*}
F(\tilde{x}) + \|\tilde{x} - v(t_0)\|^{p-1}(\tilde{x} - v(t_0)) = 0. 
\end{equation*}
In addition, we have 
\begin{equation*}
F(x(t_0)) + \|x(t_0) - v(t_0)\|^{p-1}(x(t_0) - v(t_0)) = 0. 
\end{equation*}
Combining these two equations with the previous argument implies that $\tilde{x} = x(t_0)$ which is a contradiction. Thus, $x(\cdot)$ is continuous. 

Finally, we have 
\begin{equation}\label{GEU:inequality-third}
v(t) - (I + \lambda(t)F)^{-1}v(t) = \lambda(t) F(x(t)) = \tfrac{F(x(t))}{\|F(x(t))\|^{1-1/p}}. 
\end{equation}
Combining Eq.~\eqref{GEU:inequality-first} and Eq.~\eqref{GEU:inequality-third} implies that 
\begin{equation*}
\|F(x(t))\|^{1/p} \geq \|F(x_0)\|^{1/p}e^{-\frac{t}{p-1}}, \quad \textnormal{for all } t \geq 0. 
\end{equation*}
which is equivalent to 
\begin{equation*}
\|F(x(t))\| \geq \|F(x_0)\| e^{-\frac{pt}{p-1}}, \quad \textnormal{for all } t \geq 0. 
\end{equation*}
This completes the proof. 
\end{proof}
\begin{remark}
Theorem~\ref{Thm:Global-Existence-Uniquess} shows that $F(x(t)) \neq 0$ for all $t \geq 0$ which is equivalent to the assertion that the orbit $x(\cdot)$ stays in $\Omega = \{x \in \HCal \mid F(x) \neq 0\}$. In other words, if $x(0) = x_0 \in \Omega$ and $s(0) = 0$ hold, the accelerated rescaled gradient systems in Eq.~\eqref{sys:main} can not be stabilized in finite time, which implies the asymptotic convergence behavior of corresponding discrete-time methods to an optimal solution (see Sections~\ref{sec:algorithm} and~\ref{sec:discrete} for details). In contrast, other existing rescaled gradient systems associated with convex optimization methods~\citep{Cortes-2006-Finite, Romero-2020-Finite} can exhibit finite-time convergence to an optimal solution when the function satisfies the generalized Polyak-$\L$ojasiewicz (PL) inequality. 
\end{remark}
We present our results on the asymptotic weak convergence of solution trajectories and further establish asymptotic strong convergence results under additional conditions.
\begin{theorem}\label{Thm:Rate-Asymptotic}
Suppose that $(x, v, s): [0, +\infty) \mapsto \HCal \times \HCal \times \HCal$ is a global solution of accelerated rescaled gradient system with a fixed value of $p \geq 1$. Then, there exists points $\bar{x} \in \{x \in \HCal \mid F(x) = 0\}$ such that the solution trajectories $x(t)$ and $v(t)$ weakly converge to $\bar{x}$ as $t \rightarrow +\infty$. In addition, strong convergence results can be established if either of the following conditions holds true:
\begin{enumerate}
\item $F = \nabla f$ where $f: \HCal \mapsto \br \cup \{+\infty\}$ is convex, differentiable and inf-compact.
\item $\{x \in \HCal \mid F(x) = 0\}$ has a nonempty interior. 
\end{enumerate} 
\end{theorem}
\begin{proof}
By~\citet[Theorem~3.1 and 3.6]{Lin-2023-Monotone},  the unique global solution $(v, \lambda)$ of the closed-loop control system in Eq.~\eqref{Prelim:Equal-second} satisfies the condition that there exists $\bar{x} \in \{x \in \HCal \mid F(x) = 0\}$ such that the trajectory $v(t)$ weakly converges to $\bar{x}$ and $\|\dot{v}(t)\| \rightarrow 0$ as $t \rightarrow +\infty$. Since Eq.~\eqref{sys:main} can be reformulated as Eq.~\eqref{Prelim:Equal-second} with the same $v(t)$, there exist points $\bar{x} \in \{x \in \HCal \mid F(x) = 0\}$ such that the solution trajectory $v(t)$ weakly converges to $\bar{x}$ as $t \rightarrow +\infty$ and the strong convergence results hold under the additional conditions stated in the theorem. Further, we see from Eq.~\eqref{Prelim:Equal-second} that $\|x(t) - v(t)\| = \|\dot{v}(t)\| \rightarrow 0$. This implies that the solution trajectory $x(t)$ weakly converges to the same $\bar{x}$ as $t \rightarrow +\infty$ with strong convergence under the same additional conditions. 
\end{proof}
\begin{remark}
In an infinite-dimensional setting,  the weak convergence of $x(\cdot)$ and $v(\cdot)$ to some points $\bar{x} \in \{x \in \HCal \mid F(x) = 0\}$ in Theorem~\ref{Thm:Rate-Asymptotic} is the best possible result we can expect without any additional conditions. However, strong convergence results are more desirable since it guarantees that the distance between the solution trajectories and $\bar{x}$ converges to 0 in terms of a norm~\citep{Bauschke-2001-Weak}. Indeed,~\citet{Guler-1991-Convergence} provided an example showing the importance of strong convergence where the sequence of function values converges faster if the generated iterates achieves strong convergence. In addition, the conditions imposed by Theorem~\ref{Thm:Rate-Asymptotic} are not restrictive and are verifiable in practice. 
\end{remark}
Before stating the main theorems on the global convergence rate, we define \textit{the residue function} for the ME problems. Recalling that $x_0 \in \Omega = \{x \in \HCal \mid F(x) \neq 0\}$ is an initial point, we define the residue function directly from the ME problem as follows:
\begin{equation}\label{def:residue-function}
\textsc{res}(x) = \|F(x)\|. 
\end{equation}
In the context of ME problems, it is arguable that this residue function serves as a more suitable optimality criterion than the celebrated merit function which was introduced by~\citet{Nesterov-2007-Dual}, generalizing the function gap in convex optimization problems~\citep{Facchinei-2007-Finite}. 

We are now ready to present our main results on the global convergence rate estimation in terms of the residue function in Eq.~\eqref{def:residue-function}. 
\begin{theorem}\label{Thm:Rate-Nonasymptotic}
Suppose that $(x, v, s): [0, +\infty) \mapsto \HCal \times \HCal \times \HCal$ is a global solution of the accelerated rescaled gradient system with a fixed $p \geq 1$. Then, we have
\begin{equation*}
\textsc{res}(x(t)) = O(t^{-\frac{p}{2}}). 
\end{equation*}
\end{theorem}
\begin{remark}
Theorem~\ref{Thm:Rate-Nonasymptotic} is a continuous-time version of some existing results for first-order methods. Indeed, the last-iterate convergence rate of $O(k^{-1/2})$ is achieved by extragradient method and optimistic gradient method~\citep{Gorbunov-2022-Extragradient, Cai-2022-Finite}. Although the advantage of averaging has been well recognized in the literature~\citep{Bruck-1977-Weak, Lions-1978-Methode, Nemirovski-1978-Cesari, Nemirovski-1981-Effective} and the averaged iterates can gain better rate of $O(k^{-1})$~\citep{Ouyang-2021-Lower}, such last-iterate global convergence rate results have received much attention in the recent years~\citep{Mertikopoulos-2019-Learning, Lin-2020-Finite, Golowich-2020-Tight, Golowich-2020-Last, Ba-2024-Doubly}. 
\end{remark}
We define a Lyapunov function for the system in Eq.~\eqref{sys:main} as follows: 
\begin{equation}\label{def:Lyapunov}
\ECal(t) = \|s(t)\|^2.
\end{equation}
Note that the function in Eq.~\eqref{def:Lyapunov} is the squared norm of $s(t)$ which can be seen as a continuous-time version of a dual variable in the dual extrapolation methods. This function is simpler than that used for analyzing the convergence of Newton-like inertial systems~\citep{Attouch-2011-Continuous, Attouch-2013-Global, Abbas-2014-Newton, Attouch-2020-Newton,Attouch-2021-Continuous} and closed-loop control systems~\citep{Attouch-2016-Dynamic, Lin-2023-Monotone}. We also note that our discrete-time analysis for high-order methods largely depends on the discrete-time version of the Lyapunov function in Eq.~\eqref{def:Lyapunov} with a small modification to handle the constraint sets (see Section~\ref{sec:discrete}). 
\paragraph{Proof of Theorem~\ref{Thm:Rate-Nonasymptotic}:} Using the definition of $\ECal(\cdot)$ in Eq.~\eqref{def:Lyapunov}, we have
\begin{equation}\label{inequality:nonasymptotic-first}
\frac{d\ECal(t)}{dt} = 2\langle \dot{s}(t), s(t)\rangle = -\frac{2\langle F(x(t)), v(t) - x_0\rangle}{\|F(x(t))\|^{1 - 1/p}}. 
\end{equation}
Expanding this equality with any $z \in \HCal$, we have
\begin{equation*}
\frac{d\ECal(t)}{dt} = \frac{2(\langle F(x(t)), x_0 - z\rangle + \langle F(x(t)), z - x(t)\rangle + \langle F(x(t)), x(t) - v(t)\rangle)}{\|F(x(t))\|^{1 - 1/p}}.
\end{equation*}
Since $\dot{s}(t) = - \tfrac{F(x(t))}{\|F(x(t))\|^{1-1/p}}$, we have
\begin{equation*}
\frac{\langle F(x(t)), x_0 - z\rangle}{\|F(x(t))\|^{1 - 1/p}} = \langle \dot{s}(t), x_0 - z\rangle. 
\end{equation*}
Since $F$ is monotone, we have
\begin{equation*}
\langle F(x(t)), z - x(t)\rangle \leq \langle F(z), z - x(t)\rangle. 
\end{equation*}
Since $x(t) - v(t) + \tfrac{F(x(t))}{\|F(x(t))\|^{1-1/p}} = 0$, we have
\begin{equation*}
\frac{\langle F(x(t)), x(t) - v(t)\rangle}{\|F(x(t))\|^{1 - 1/p}} = -\|F(x(t))\|^{\frac{2}{p}}. 
\end{equation*}
Putting these pieces together in Eq.~\eqref{inequality:nonasymptotic-first} yields that, for any $z \in \HCal$, we have
\begin{equation}\label{inequality:nonasymptotic-second}
\frac{d\ECal(t)}{dt} \leq 2\langle \dot{s}(t), x_0 - z\rangle - 2\|F(x(t))\|^{\frac{1-p}{p}}\langle F(z), x(t) - z\rangle - 2\|F(x(t))\|^{\frac{2}{p}}. 
\end{equation}
Letting $z = \bar{x}$ be a solution satisfying $F(\bar{x}) = 0$ and $\|x_0 - \bar{x}\| \leq D$ in Eq.~\eqref{inequality:nonasymptotic-second}, we have
\begin{equation*}
\frac{d\ECal(t)}{dt} \leq 2\langle \dot{s}(t), x_0 - \bar{x}\rangle - 2\|F(x(t))\|^{\frac{2}{p}}. 
\end{equation*}
Integrating this inequality over $[0, t]$ and using $\ECal(0) = 0$, we have
\begin{equation*}
\int_0^t \|F(x(s))\|^{\frac{2}{p}} \; ds \leq \frac{1}{2}\left(\langle s(t), x_0 - z\rangle - \|s(t)\|^2\right) \leq \frac{D^2}{8}, \quad \textnormal{for all } t \geq 0.  
\end{equation*}
We claim that $t \mapsto \|F(x(t))\|$ is nonincreasing. Indeed, for the case of $p=1$, we have 
\begin{equation*}
\dot{s}(t) = -F(x(t)), \quad v(t) = x_0 + s(t), \quad x(t) - v(t) + F(x(t)) = 0.
\end{equation*}
In this case, we can write $x(t) = (I + F)^{-1}v(t)$. Since $v(\cdot)$ is continuously differentiable, we have $x(\cdot)$ is also continuously differentiable. Define the function $g(t) = \frac{1}{2}\|v(t) - x(t)\|^2$. Then, we have 
\begin{equation*}
\dot{g}(t) = \langle \dot{v}(t) - \dot{x}(t), v(t) - x(t) \rangle = -\langle \dot{v}(t) - \dot{x}(t), \dot{v}(t) \rangle = -\|\dot{v}(t) - \dot{x}(t)\|^2 - \langle \dot{v}(t) - \dot{x}(t), \dot{x}(t) \rangle. 
\end{equation*}
Since $\dot{v}(t) - \dot{x}(t) = F(\dot{x}(t))$ and $F$ is monotone, we have $\langle \dot{v}(t) - \dot{x}(t), \dot{x}(t) \rangle \geq 0$. This implies that $\dot{g}(t) \leq 0$ and hence $g(t)$ is nonincreasing. For the case of $p \geq 2$, Theorem~\ref{Thm:Global-Existence-Uniquess} implies that $\|F(x(t))\|^{1/p-1}$ is nondecreasing. Since $1/p-1 < 0$, we have $t \mapsto \|F(x(t))\|$ is nonincreasing as desired. Putting these pieces together yields that 
\begin{equation*}
\textsc{res}(x(t)) = \|F(x(t))\| \leq \left(\frac{D^2}{8t}\right)^{\frac{p}{2}} = O(t^{-\frac{p}{2}}). 
\end{equation*}
This completes the proof. 

\subsection{Discussion}
In continuous-time optimization, a dynamical system is designed to be computable under the oracles of a convex and differentiable function $f: \HCal \mapsto \br$ such that the solution trajectory $x(t)$ converges to a minimizer of $f$. A classical example is the gradient flow in the form of $\dot{x}(t) + \nabla f(x(t)) = 0$ which has been long studied due to its ability to yield that $x(t)$ converges to a minimizer of $f$ as $t \rightarrow +\infty$~\citep{Hadamard-1908-Memoire, Courant-1943-Variational, Ambrosio-2005-Gradient}. The subsequent work of~\citet{Schropp-1995-Using} and~\citet{Schropp-2000-Dynamical} explored links between nonlinear dynamical systems and gradient-based optimization, including nonlinear constraints.

In his seminal work,~\citet{Cortes-2006-Finite} proposed two \textit{discontinuous} normalized modifications of gradient flows to attain \textit{finite-time convergence}.  Specifically, his dynamical systems are 
\begin{equation}\label{sys:NGF-first}
\dot{x}(t) + \frac{\nabla f(x(t))}{\|\nabla f(x(t))\|} = 0, 
\end{equation}
and 
\begin{equation}\label{sys:NGF-second}
\dot{x}(t) + \sign(\nabla f(x(t))) = 0. 
\end{equation}
If $f$ is twice continuously differentiable and strongly convex in an open neighborhood $D \subseteq \br^n$ of $x^\star$, he proved that the maximal solutions to the dynamical systems in Eq.~\eqref{sys:NGF-first} and~\eqref{sys:NGF-second} (in the sense of Filippov~\citep{Filippov-1964-Differential, Paden-1987-Calculus, Arscott-1988-Differential}) exist and converge in finite time to $x^\star$, given that $x_0$ is in some positively invariant compact subset $S \subseteq D$. The convergence times are upper bounded by
\begin{equation*}
t^\star \leq \left\{\begin{array}{lr}
\tfrac{\|\nabla f(x_0)\|}{\min_{x \in S} \lambda_{\min}[\nabla^2 f(x)]}, & \textnormal{ for Eq.~\eqref{sys:NGF-first}}, \\ 
\tfrac{\|\nabla f(x_0)\|_1}{\min_{x \in S} \lambda_{\min}[\nabla^2 f(x)]}, & \textnormal{ for Eq.~\eqref{sys:NGF-second}}. 
\end{array}\right. 
\end{equation*}
Recently,~\citet{Romero-2020-Finite} have extended the finite-time convergence results to the rescaled gradient flow proposed by~\citet{Wibisono-2016-Variational} as follows:  
\begin{equation}\label{sys:RGF-first}
\dot{x}(t) + \frac{\nabla f(x(t))}{\|\nabla f(x(t))\|^{1-1/p}} = 0, 
\end{equation}
and considered a new rescaled gradient flow as follows, 
\begin{equation}\label{sys:RGF-second}
\dot{x}(t) + \|\nabla f(x(t))\|^{1/p}\sign(\nabla f(x(t))) = 0. 
\end{equation}
In particular, if $f$ is continuously differentiable and $\mu$-gradient-dominated of order $q \in (1, p+1)$ near a strict local minimizer $x^\star$, the maximal solutions to the dynamical systems in Eq.~\eqref{sys:RGF-first} and~\eqref{sys:RGF-second} (in the sense of Filippov) exist and converge in finite time to $x^\star$,  given that $\|x_0 - x^\star\| > 0$ is sufficiently small. The upper bounds for convergence times are summarized in~\citet[Theorem~1]{Romero-2020-Finite}. 

However, the above works are restricted to studying the dynamical systems associated with convex optimization methods. It has been unknown whether or not these methodologies can be extended to monotone equation problems and further lead to accelerated first-order and high-order methods with global convergence rate guarantees.

%% file: sec/algorithm.tex
%!TEX root = paper.tex
\section{Discrete-Time Algorithms}\label{sec:algorithm}
We propose two algorithmic frameworks that arise via temporal discretization of the systems in Eq.~\eqref{sys:main}. Our approach highlights the importance of the dual extrapolation step~\citep{Nesterov-2007-Dual}, interpreting it as the discretization of an integral. The first framework together with an approximate tensor subroutine recovers the existing $p^\textnormal{th}$-order VI methods specialized to solve ME problems while the second framework yields a new suite of simple first-order methods for solving ME problems. 

\subsection{Temporal discretization} 
We begin by studying an algorithm which is derived by temporal discretization of the accelerated rescaled gradient systems in Eq.~\eqref{sys:main} in a Euclidean setting. In particular, the systems can be rewritten as follows, 
\begin{equation*}
\begin{array}{lll}
\dot{s}(t) = - \tfrac{F(x(t))}{\|F(x(t))\|^{1-1/p}}, & v(t) = x_0 + s(t), & x(t) - v(t) + \tfrac{F(x(t))}{\|F(x(t))\|^{1-1/p}} = 0. 
\end{array} 
\end{equation*}
By introducing $\lambda(t) = \|F(x(t))\|^{1/p-1}$, we obtain
\begin{equation*}
\left\{\begin{array}{ll}
& \dot{s}(t) + \lambda(t)F(x(t)) = 0, \\
& v(t) = x_0 + s(t), \\ 
& x(t) - v(t) + \tfrac{F(x(t))}{\|F(x(t))\|^{1-1/p}} = 0, \\
& (\lambda(t))^p\|F(x(t))\|^{p-1} = 1. \\ 
\end{array}\right.  
\end{equation*}
Since $\|x(t) - v(t)\| = \|F(x(t))\|^{1/p} = \|\lambda(t)F(x(t))\|$, we have
\begin{equation}\label{Prelim:Equal-third}
\left\{\begin{array}{ll}
& \dot{s}(t) + \lambda(t)F(x(t)) = 0, \\
& - s(t) + v(t) - x_0 = 0, \\ 
& F(x(t)) + \|x(t) - v(t)\|^{p-1}(x(t) - v(t)) = 0, \\
& \lambda(t)\|x(t) - v(t)\|^{p-1} = 1. 
\end{array}\right.  
\end{equation}
We define a sequence $\{(x_k, v_{k+1}, s_k, \lambda_{k+1})\}_{k \geq 0}$ as a discrete-time counterpart of the solution trajectory $\{(x(t), v(t), s(t), \lambda(t))\}_{t \geq 0}$. By appeal to a discretization of the systems in Eq.~\eqref{Prelim:Equal-third}, we have
\begin{equation}\label{sys:discrete}
\left\{\begin{array}{ll}
& s_{k+1} - s_k + \lambda_{k+1} F(x_{k+1}) = 0, \\ 
& - s_k + v_{k+1} - x_0 = 0, \\
& F(x_{k+1}) + \|x_{k+1} - v_{k+1}\|^{p-1}(x_{k+1} - v_{k+1}) = 0, \\
& \lambda_{k+1}\|x_{k+1} - v_{k+1}\|^{p-1} = 1, \\
& x_0 \in \{x \in \br^d \mid F(x) \neq 0\} \textnormal{ and } s_0 = 0. 
\end{array}\right.
\end{equation}
According to the argument from~\citet{Lin-2022-Control}, we relax the equation $\lambda_{k+1}\|x_{k+1} - v_{k+1}\|^{p-1} = 1$ using $\lambda_{k+1}\|x_{k+1} - v_{k+1}\|^{p-1} \geq \theta$, where $\theta > 0$ is a parameter that enhances the flexibility of our methods. It is worth mentioning that the aforementioned frameworks are intrinsically different from the large-step A-HPE framework of~\citet{Monteiro-2012-Iteration} and its high-order extension~\citep{Lin-2023-Monotone} in that computing $x_{k+1}$ in our new frameworks do not require the value of $\lambda_{k+1}$. Indeed, the large-step A-HPE framework and its variants need to compute a pair of $x_{k+1} \in \br^d$ and $\lambda_{k+1} > 0$ jointly since the computation of $x_{k+1}$ is based on $\lambda_{k+1}$ and the computation of $\lambda_{k+1}$ is based on $x_{k+1}$. This is the reason why the binary search is necessary in those frameworks. 

\subsection{Algorithmic schemes} 
By instantiating our new framework with an approximate tensor subroutine, we can recover the existing $p^{\textnormal{th}}$-order VI methods when specialized to solve ME problems~\citep{Lin-2022-Perseus}. Indeed, the idea behind our approach is that we consider the setting where $F$ is $p^{\textnormal{th}}$-order Lipschitz continuous and define the approximate tensor subroutine based on the Taylor expansion of $F$. 
\begin{table}[!t]
\begin{tabular}{cc}
\begin{minipage}{.48\textwidth}
\begin{algorithm}[H]\small
\begin{algorithmic}\caption{High-Order ME Methods}\label{Algorithm:HO}
\STATE \textbf{Input:} $p \geq 1$, $x_0 \in \br^d$, $L > 0$ and $T \geq 1$. 
\STATE \textbf{Initialization:} $s_0 = 0 \in \br^d$.
\FOR{$k = 0, 1, 2, \ldots, T$} 
\STATE If $x_k$ is a solution of Eq.~\eqref{prob:main}, then \textbf{Stop}. 
\STATE $v_{k+1} \leftarrow x_0 + s_k$. 
\STATE Find $x_{k+1}$ satisfying $F_{v_{k+1}}(x_{k+1}) = 0$. 
\STATE $\lambda_{k+1} \leftarrow \frac{\eta_1(p, L)}{\|x_{k+1} - v_{k+1}\|^{p-1}}$ where $\eta_1$ is a function. 
\STATE $s_{k+1} \leftarrow s_k - \lambda_{k+1} F(x_{k+1})$. 
\ENDFOR
\STATE \textbf{Output:} $\bar{x} = \argmin_{x \in \{x_0, \ldots, x_T\}} \|F(x)\|$. 
\end{algorithmic}
\end{algorithm} 
\end{minipage} &
\begin{minipage}{.48\textwidth}
\begin{algorithm}[H]\small
\begin{algorithmic}\caption{First-Order ME Methods}\label{Algorithm:FO}
\STATE \textbf{Input:} $p \geq 1$, $x_0 \in \br^d$, $L > 0$ and $T \geq 1$. 
\STATE \textbf{Initialization:} $s_0 = 0 \in \br^d$.
\FOR{$k = 0, 1, 2, \ldots, T$} 
\STATE If $x_k$ is a solution of Eq.~\eqref{prob:main}, then \textbf{Stop}. 
\STATE $v_{k+1} \leftarrow x_0 + s_k$. 
\STATE $x_{k+1} \leftarrow v_{k+1} - \gamma\|F(v_{k+1})\|^{1/p-1}F(v_{k+1})$. 
\STATE $\lambda_{k+1} \leftarrow \frac{\eta_2(p, L, \gamma)}{\|x_{k+1} - v_{k+1}\|^{p-1}}$ where $\eta_2$ is a function. 
\STATE $s_{k+1} \leftarrow s_k - \lambda_{k+1} F(x_{k+1})$. 
\ENDFOR
\STATE \textbf{Output:} $\bar{x} = \argmin_{x \in \{x_0, \ldots, x_T\}} \|F(x)\|$. 
\end{algorithmic}
\end{algorithm} 
\end{minipage}
\end{tabular}
\end{table}

We can see from our framework that the following step plays a pivotal role: 
\begin{equation*}
\textnormal{Find } x_{k+1} \in \br^d \textnormal{ such that } F(x_{k+1}) + \|x_{k+1} - v_{k+1}\|^{p-1}(x_{k+1} - v_{k+1}) = 0,
\end{equation*}
which amounts to solving a strictly monotone equation problem and which could be challenging from a computational point of view. If $F$ is $p^{\textnormal{th}}$-order Lipschitz continuous, we can approximate $F(x_{k+1})$ using the Taylor expansion of $F$ at a point $v \in \br^d$. That is, we define
\begin{equation}\label{def:approximation}
F_v(x) = F(v) + \langle DF(v), x-v\rangle + \ldots + \tfrac{1}{(p-1)!}D^{(p-1)} F(v)[x-v]^{p-1} + \tfrac{2L}{(p-1)!}\|x - v\|^{p-1}(x - v), 
\end{equation}
and consider the following subproblem given by
\begin{equation*}
\textnormal{Find } x_{k+1} \in \br^d \textnormal{ such that } F_{v_{k+1}}(x_{k+1}) = 0. 
\end{equation*}
The resulting $p^{\textnormal{th}}$-order methods recover the existing $p^{\textnormal{th}}$-order VI methods when specialized to solve ME problems~\citep{Lin-2022-Perseus} and we summarize the details in Algorithm~\ref{Algorithm:HO}. For simplicity, we assume an access to an oracle which returns an exact solution of each subproblem and do not take the inexact approaches into account. 

Following the idea of~\citet{Wilson-2019-Accelerating}, we also approximate $F(x_{k+1})$ using the first-order Taylor expansion of $F$ at a point $v \in \br^d$ and consider another subproblem given by
\begin{equation*}
\textnormal{Find } x_{k+1} \in \br^d \textnormal{ such that } F(v_{k+1}) + \tfrac{1}{\gamma^p}\|x_{k+1} - v_{k+1}\|^{p-1}(x_{k+1} - v_{k+1}) = 0. 
\end{equation*}
This yields a new suite of simple first-order methods where each subproblem admits a closed-form solution. We summarize the details in Algorithm~\ref{Algorithm:FO}. 

We also summarize the restarted versions of two above ME methods in Algorithm~\ref{Algorithm:Restart}, which is built upon the restart strategy. Suppose that $F$ is strongly monotone around a unique optimal solution and $p \geq 2$, we consider Algorithm~\ref{Algorithm:Restart} and prove that it achieves a local superlinear convergence. Concerning Algorithm~\ref{Algorithm:Restart}, we run one iteration of Algorithm~\ref{Algorithm:HO} and~\ref{Algorithm:FO} initialized with $x_k$ at each epoch; indeed, we simply restart Algorithm~\ref{Algorithm:HO} and~\ref{Algorithm:FO} properly. It is also promising to combine our methods with adaptive restart strategies~\citep{Donoghue-2015-Adaptive, Roulet-2017-Sharpness}, and/or prove a local convergence guarantee under weaker conditions than the strong monotonicity of $F$. 
\begin{algorithm}[!t]
\begin{algorithmic}\caption{The Restart Versions of Algorithms~\ref{Algorithm:HO} and~\ref{Algorithm:FO}}\label{Algorithm:Restart}
\STATE \textbf{Input:} $p \geq 1$, $x_0 \in \br^d$, $L > 0$ and $T \geq 1$. 
\FOR{$k = 0, 1, 2, \ldots, T$}
\STATE If $x_k$ is a solution of Eq.~\eqref{prob:main}, then \textbf{Stop}. 
\STATE $x_{k+1}$ is an output of Algorithms~\ref{Algorithm:HO} and~\ref{Algorithm:FO} with $(p, x_k, L, 1)$, i.e., one iteration but initialized with $x_k$. 
\ENDFOR
\end{algorithmic}
\end{algorithm}

\subsection{Main results}
We start with the presentation of our main results for Algorithm~\ref{Algorithm:HO} and~\ref{Algorithm:FO}. To facilitate the presentation, we rewrite the definition of a residue function as $\textsc{res}(x) = \|F(x)\|$ and also recall the definitions of $p^{\textnormal{th}}$-order Lipschitz continuous and $p^{\textnormal{th}}$-order strongly Lipschitz continuous operators as follows.
\begin{definition}
$F$ is $p^{\textnormal{th}}$-order Lipschitz continuous if $\|D^{(p-1)} F(x') - D^{(p-1)} F(x)\| \leq L\|x' - x\|$. 
\end{definition}
\begin{definition}
$F$ is $p^{\textnormal{th}}$-order strongly Lipschitz continuous if $\|D^{(p-1)} F(x') - D^{(p-1)} F(x)\| \leq L\|x' - x\|$ and $\|D^{(m)}F(x)[F(x)]^m\| \leq L\|F(x)\|^{m+1-m/p}$ for $1 \leq m \leq p-1$. 
\end{definition}
We summarize our results in the following theorems.
\begin{theorem}\label{Thm:Global-Rate-High}
Letting $F$ be monotone and $p^{\textnormal{th}}$-order Lipschitz continuous and we set $\frac{p!}{(12p-6)L} \leq \eta_1(p, L) \leq \frac{p!}{(4p+2)L}$ in Algorithm~\ref{Algorithm:HO}. Then, the iterates $\{x_k\}_{k \geq 0}$ generated by the algorithm satisfy 
\begin{equation*}
\inf_{0 \leq i \leq k} \textsc{res}(x_i) = O(k^{-\frac{p}{2}}). 
\end{equation*}
\end{theorem}
\begin{theorem}\label{Thm:Global-Rate-First}
Letting $F$ be monotone and $p^{\textnormal{th}}$-order strongly Lipschitz continuous and we set $\frac{\gamma^p}{6-6\gamma Lc_p} \leq \eta_2(p, L, \gamma) \leq \frac{\gamma^p}{2+2\gamma Lc_p}$ where $0 < \gamma < \min\{1, \frac{1}{2Lc_p}\}$ with $c_p = \sum_{m=1}^p \tfrac{1}{m!}$ in Algorithm~\ref{Algorithm:FO}. Then, the iterates $\{x_k\}_{k \geq 0}$ generated by the algorithm satisfy 
\begin{equation*}
\inf_{0 \leq i \leq k} \textsc{res}(x_i) = O(k^{-\frac{p}{2}}). 
\end{equation*}
\end{theorem}
The local convergence guarantee for Algorithm~\ref{Algorithm:Restart} when $p \geq 2$ was established in the following theorem. Our analysis is based on Theorem~\ref{Thm:Global-Rate-High} and~\ref{Thm:Global-Rate-First} and our results are expressed in terms of $\|x_k - x^\star\|$, where $x^\star \in \br^d$ is a unique solution of Eq.~\eqref{prob:main} in the sense that $F(x^\star) = 0$. 
\begin{theorem}\label{Thm:Local-Rate-Estimate}
Suppose that $F$ be $\mu$-strongly monotone. Then, the following statements hold true.
\begin{itemize}
\item Under the same assumption as in Theorem~\ref{Thm:Global-Rate-High}, the iterates $\{x_k\}_{k \geq 0}$ generated by Algorithm~\ref{Algorithm:Restart} with Algorithm~\ref{Algorithm:HO} as the subroutine satisfy
\begin{equation*}
\|x_{k+1} - x^\star\| \leq \left(\tfrac{4^p(2p+1)}{p!}\tfrac{L}{\mu}\right)\|x_k - x^\star\|^p, \quad \textnormal{for all } k \geq 0. 
\end{equation*}
As a consequence, the above iterates $\{x_k\}_{k \geq 0}$ converge with a rate at least of the order $p \geq 2$ if the following condition holds true:
\begin{equation*}
\|x_0 - x^\star\| \leq \tfrac{1}{2}\left(\tfrac{p!}{4^p(2p+1)}\tfrac{\mu}{L}\right)^{\frac{1}{p-1}}. 
\end{equation*}
\item Under the same assumption as in Theorem~\ref{Thm:Global-Rate-First}, the iterates $\{x_k\}_{k \geq 0}$ generated by Algorithm~\ref{Algorithm:Restart} with Algorithm~\ref{Algorithm:FO} as the subroutine satisfy
\begin{equation*}
\|x_{k+1} - x^\star\| \leq \left(\tfrac{2^{2p+1}}{\mu\gamma^p}\right)\|x_k - x^\star\|^p, \quad \textnormal{for all } k \geq 0. 
\end{equation*}
As a consequence, the above iterates $\{x_k\}_{k \geq 0}$ converge with a rate at least of the order $p \geq 2$ if the following condition holds true:
\begin{equation*}
\|x_0 - x^\star\| \leq \tfrac{1}{2}\left(\tfrac{\mu\gamma^p}{2^{2p+1}}\right)^{\frac{1}{p-1}}. 
\end{equation*}
\end{itemize}
\end{theorem}
\begin{remark}
Theorem~\ref{Thm:Global-Rate-High} improves the convergence rate of $O(k^{-(p+1)/2}\log(k))$ which was obtained for $p^{\textnormal{th}}$-order line-search-based methods~\citep{Lin-2023-Monotone} under the same assumptions. The local convergence guarantee in Theorem~\ref{Thm:Local-Rate-Estimate} has been derived by~\citet{Jiang-2022-Generalized} for generalized optimistic gradient methods but without counting the required number of binary searches at each iteration. 
\end{remark}
\begin{remark}
Theorem~\ref{Thm:Global-Rate-First} and~\ref{Thm:Local-Rate-Estimate} are new and can be understood as a generalization of~\citet{Wilson-2019-Accelerating} from unconstrained convex optimization problems to monotone equation problems. There are also examples of strongly Lipschitz continuous operators in machine learning and we refer to~\citet{Wilson-2019-Accelerating} for concrete examples and relevant discussion. 
\end{remark}
\begin{remark}
The anchored extragradient method from~\citet{Yoon-2021-Accelerated} was shown to achieve an optimal convergence rate of $O(1/k)$ in terms of $\|F(x)\|$ and thus outperforms Algorithm~\ref{Algorithm:HO} for the case of $p=1$. Can we interpret the anchoring scheme from a continuous-time viewpoint and develop a high-order anchored method with an improved convergence rate guarantee? We leave this for future work and believe this would be a worthwhile research problem that deserves a specialized treatment on its own.
\end{remark}

%% file: sec/analysis.tex
%!TEX root = paper.tex
\section{Discrete-Time Convergence Analysis}\label{sec:discrete}
We conduct the discrete-time convergence analysis for Algorithm~\ref{Algorithm:HO} and~\ref{Algorithm:FO}. In terms of a residue function, we prove a global rate of $O(k^{-p/2})$ for Algorithm~\ref{Algorithm:HO} if $F$ is $p^\textnormal{th}$-order Lipschitz continuous and Algorithm~\ref{Algorithm:FO} if $F$ is $p^\textnormal{th}$-order strongly Lipschitz continuous. Our analysis is motivated by the continuous-time analysis from Section~\ref{sec:system} and thus differs from the previous analysis in~\citet{Lin-2022-Perseus}. 

\subsection{Proof of Theorem~\ref{Thm:Global-Rate-High}}
We construct a discrete-time Lyapunov function for Algorithm~\ref{Algorithm:HO} as follows: 
\begin{equation}\label{def:Lyapunov-discrete}
E_k = \tfrac{1}{2}\|s_k\|^2,  
\end{equation}
which will be used to prove several technical lemmas for proving Theorem~\ref{Thm:Global-Rate-High}. 
\begin{lemma}\label{Lemma:HO-descent}
For every integer $T \geq 1$, we have
\begin{equation*}
\sum_{k=1}^T \lambda_k \langle F(x_k), x_k - x\rangle \leq E_0 - E_T + \langle s_T, x - x_0\rangle -\tfrac{1}{10}\left(\sum_{k=1}^T \|x_k - v_k\|^2\right), \quad \textnormal{for all } x \in \br^d.  
\end{equation*}
\end{lemma}
\begin{proof}
We have 
\begin{equation*}
E_{k+1} - E_k = \langle s_{k+1} - s_k, s_{k+1}\rangle - \tfrac{1}{2}\|s_{k+1} - s_k\|^2.
\end{equation*}
Combining this equation with the definition of $v_{k+1}$, we have
\begin{equation*}
E_{k+1} - E_k = \lambda_{k+1}\langle F(x_{k+1}), x_0 - v_{k+2}\rangle - \tfrac{1}{2}\|v_{k+2} - v_{k+1}\|^2. 
\end{equation*}
Letting $x \in \br^d$, we have
\begin{equation*}
E_{k+1} - E_k \leq \lambda_{k+1} \langle F(x_{k+1}), x_0 - x\rangle + \lambda_{k+1} \langle F(x_{k+1}), x - x_{k+1}\rangle + \lambda_{k+1} \langle F(x_{k+1}), x_{k+1} - v_{k+2}\rangle - \tfrac{1}{2}\|v_{k+2} - v_{k+1}\|^2. 
\end{equation*}
Summing up this inequality over $k = 0, 1, \ldots, T-1$ and changing the counter $k+1$ to $k$ yields 
\begin{equation}\label{inequality:HO-descent-first}
\sum_{k=1}^T \lambda_k \langle F(x_k), x_k - x\rangle \leq E_0 - E_T + \underbrace{\sum_{k=1}^T \lambda_k \langle F(x_k), x_0 - x\rangle}_{\textbf{I}} + \underbrace{\sum_{k=1}^T \lambda_k \langle F(x_k), x_k - v_{k+1}\rangle - \tfrac{1}{2}\|v_k - v_{k+1}\|^2}_{\textbf{II}}. 
\end{equation}
Using the update formula $s_{k+1} = s_k - \lambda_{k+1} F(x_{k+1})$ and $s_0 = 0 \in \br^d$, we have
\begin{equation}\label{inequality:HO-descent-second}
\textbf{I} = \langle s_0 - s_T, x_0 - x\rangle = \langle s_T, x - x_0\rangle.  
\end{equation}
Recall that $F_{v_k}(x_k) = 0$ where $F_v(x): \br^d \rightarrow \br^d$ is defined with any fixed $v \in \XCal$ as follows, 
\begin{equation*}
F_v(x) = F(v) + \langle DF(v), x-v\rangle + \ldots + \tfrac{1}{(p-1)!}D^{(p-1)} F(v)[x-v]^{p-1} + \tfrac{2L}{(p-1)!}\|x - v\|^{p-1}(x - v). 
\end{equation*}
Since $F$ is $p^{\textnormal{th}}$-order Lipschitz continuous, we have 
\begin{equation}\label{inequality:HO-descent-third}
\|F(x_k) - F_{v_k}(x_k) + \tfrac{2L}{(p-1)!}\|x_k - v_k\|^{p-1}(x_k - v_k)\| \leq \tfrac{L}{p!}\|x_k - v_k\|^p. 
\end{equation}
We perform a decomposition of $\langle F(x_k), x_k - v_{k+1}\rangle$ and derive from Eq.~\eqref{inequality:HO-descent-third} that 
\begin{eqnarray*}
\lefteqn{\langle F(x_k), x_k - v_{k+1}\rangle} \\
& = & \langle F(x_k) - F_{v_k}(x_k) + \tfrac{2L}{(p-1)!}\|x_k - v_k\|^{p-1}(x_k - v_k), x_k - v_{k+1}\rangle - \tfrac{2L}{(p-1)!}\|x_k - v_k\|^{p-1} \langle x_k - v_k, x_k - v_{k+1}\rangle \\
& \leq & \|F(x_k) - F_{v_k}(x_k) + \tfrac{2L}{(p-1)!}\|x_k - v_k\|^{p-1}(x_k - v_k)\|\|x_k - v_{k+1}\| - \tfrac{2L}{(p-1)!}\|x_k - v_k\|^{p-1} \langle x_k - v_k, x_k - v_{k+1}\rangle \\
& \leq & \tfrac{L}{p!}\|x_k - v_k\|^p\|x_k - v_{k+1}\| - \tfrac{2L}{(p-1)!}\|x_k - v_k\|^{p-1} \langle x_k - v_k, x_k - v_{k+1}\rangle \quad \textnormal{(cf. Eq.~\eqref{inequality:HO-descent-third})} \\
& \leq & \tfrac{L}{p!}\|x_k - v_k\|^{p+1} + \tfrac{L}{p!}\|x_k - v_k\|^p\|v_k - v_{k+1}\| - \tfrac{2L}{(p-1)!}\|x_k - v_k\|^{p-1} \langle x_k - v_k, x_k - v_{k+1}\rangle. 
\end{eqnarray*}
Note that we have
\begin{equation*}
\langle x_k - v_k, x_k - v_{k+1}\rangle = \|x_k - v_k\|^2 + \langle x_k - v_k, v_k - v_{k+1}\rangle \geq \|x_k - v_k\|^2 - \|x_k - v_k\|\|v_k - v_{k+1}\|. 
\end{equation*}
Putting these pieces together yields that 
\begin{equation*}
\langle F(x_k), x_k - v_{k+1}\rangle \leq \tfrac{(2p+1)L}{p!}\|x_k - v_k\|^p\|v_k - v_{k+1}\| - \tfrac{(2p-1)L}{p!}\|x_k - v_k\|^{p+1}. 
\end{equation*}
Since $\lambda_k = \frac{\eta_1(p, L)}{\|x_k - v_k\|^{p-1}}$ and $\frac{p!}{(12p-6)L} \leq \eta_1(p, L) \leq \frac{p!}{(4p+2)L}$, we have $\frac{1}{12p-6} \leq \frac{\lambda_k L\|x_k - v_k\|^{p-1}}{p!} \leq \frac{1}{4p+2}$. Thus, we have
\begin{eqnarray}\label{inequality:HO-descent-fourth}
\textbf{II} & \leq & \sum_{k=1}^T \left(\tfrac{(2p+1)\lambda_k L}{p!}\|x_k - v_k\|^p\|\|v_k - v_{k+1}\| - \tfrac{1}{2}\|v_k - v_{k+1}\|^2 - \tfrac{(2p-1)\lambda_k L}{p!}\|x_k - v_k\|^{p+1}\right) \nonumber \\
& \leq & \sum_{k=1}^T \left(\tfrac{1}{2}\|x_k - v_k\|\|\|v_k - v_{k+1}\| - \tfrac{1}{2}\|v_k - v_{k+1}\|^2 - \tfrac{1}{6}\|x_k - v_k\|^2\right) \nonumber \\ 
& \leq & \sum_{k=1}^T \left(\max_{a \geq 0}\left\{\tfrac{1}{2}\|x_k - v_k\|a - \tfrac{1}{2}a^2\right\} - \tfrac{1}{6}\|x_k - v_k\|^2\right) \nonumber \\
& = & -  \tfrac{1}{24} \left(\sum_{k=1}^T \|x_k - v_k\|^2 \right). 
\end{eqnarray}
Plugging Eq.~\eqref{inequality:HO-descent-second} and Eq.~\eqref{inequality:HO-descent-fourth} into Eq.~\eqref{inequality:HO-descent-first} yields that 
\begin{equation*}
\sum_{k=1}^T \lambda_k \langle F(x_k), x_k - x\rangle \leq E_0 - E_T + \langle s_T, x - x_0\rangle -\tfrac{1}{24}\left(\sum_{k=1}^T \|x_k - v_k\|^2 \right). 
\end{equation*}
This completes the proof. 
\end{proof}
\begin{lemma}\label{Lemma:HO-error}
For every integer $T \geq 1$ and let $x \in \br^d$, we have
\begin{equation*}
\sum_{k=1}^T \|x_k - v_k\|^2 \leq 12\|x^\star - x_0\|^2,
\end{equation*}
where $x^\star \in \br^d$ is a solution of Eq.~\eqref{prob:main} satisfying that $F(x^\star) = 0$. 
\end{lemma}
\begin{proof}
For any $x \in \br^d$, we have
\begin{equation*}
E_0 - E_T + \langle s_T, x - x_0\rangle = E_0 - \tfrac{1}{2}\|s_T\|^2 + \langle s_T, x - x_0\rangle. 
\end{equation*}
Since $s_0 = 0$, we have $E_0 = 0$. By Young's inequality, we have
\begin{equation*}
E_0 - E_T + \langle s_T, x - x_0\rangle \leq - \tfrac{1}{2}\|s_T\|^2 + \tfrac{1}{2}\|s_T\|^2 + \tfrac{1}{2}\|x - x_0\|^2 = \tfrac{1}{2}\|x - x_0\|^2. 
\end{equation*}
This together with Lemma~\ref{Lemma:HO-descent} yields that 
\begin{equation*}
\sum_{k=1}^T \lambda_k \langle F(x_k), x_k - x\rangle + \tfrac{1}{24}\left(\sum_{k=1}^T \|x_k - v_k\|^2\right) \leq \tfrac{1}{2}\|x - x_0\|^2, \quad \textnormal{for all } x \in \br^d, 
\end{equation*}
which implies the first inequality. Letting $x = x^\star$ be a solution of Eq.~\eqref{prob:main} satisfying that $F(x^\star) = 0$ in the above inequality yields the desired inequality. 
\end{proof}
\paragraph{Proof of Theorem~\ref{Thm:Global-Rate-High}:} We let $x^\star \in \br^d$ be a solution of Eq.~\eqref{prob:main} defined in Lemma~\ref{Lemma:HO-error} and have
\begin{equation}\label{inequality:HO-main-first}
\inf_{1 \leq k \leq T} \|x_k - v_k\|^2 \leq \tfrac{1}{T}\sum_{k=1}^T \|x_k - v_k\|^2 \leq \tfrac{12\|x^\star - x_0\|^2}{T}.  
\end{equation}
Recall that Eq.~\eqref{inequality:HO-descent-third} in the proof of Lemma~\ref{Lemma:HO-descent} yields
\begin{equation*}
\|F(x_k) - F_{v_k}(x_k) + \tfrac{2L}{(p-1)!}\|x_k - v_k\|^{p-1}(x_k - v_k)\| \leq \tfrac{L}{p!}\|x_k - v_k\|^p. 
\end{equation*}
Since $F_{v_k}(x_k) = 0$, we have
\begin{equation*}
\|F(x_k)\| \leq \tfrac{L}{p!}\|x_k - v_k\|^p + \tfrac{2L}{(p-1)!}\|x_k - v_k\|^p \leq \tfrac{(2p+1)L}{p!}\|x_k - v_k\|^p.  
\end{equation*}
By the definition of $\textsc{res}(\cdot)$, we have
\begin{equation}\label{inequality:HO-main-second}
\inf_{1 \leq k \leq T} \textsc{res}(x_k) = \inf_{1 \leq k \leq T} \|F(x_k)\| \leq \tfrac{(2p+1)L}{p!}\left(\inf_{1 \leq k \leq T}\|x_k - v_k\|^p\right). 
\end{equation}
Plugging Eq.~\eqref{inequality:HO-main-first} into Eq.~\eqref{inequality:HO-main-second} yields the desired result. 

\subsection{Proof of Theorem~\ref{Thm:Global-Rate-First}}
We use the same discrete-time Lyapunov function in Eq.~\eqref{def:Lyapunov-discrete} to analyze the dynamics of Algorithm~\ref{Algorithm:FO}. 
\begin{lemma}\label{Lemma:FO-descent}
For every integer $T \geq 1$, we have
\begin{equation*}
\sum_{k=1}^T \lambda_k \langle F(x_k), x_k - x\rangle \leq E_0 - E_T + \langle s_T, x - x_0\rangle -\tfrac{1}{10}\left(\sum_{k=1}^T \|x_k - v_k\|^2\right), \quad \textnormal{for all } x \in \br^d.  
\end{equation*}
\end{lemma}
\begin{proof}
Using the same arguments as in Lemma~\ref{Lemma:HO-descent}, we have
\begin{equation}\label{inequality:FO-descent-first}
\sum_{k=1}^T \lambda_k \langle F(x_k), x_k - x\rangle \leq E_0 - E_T + \langle s_T, x - x_0\rangle + \underbrace{\sum_{k=1}^T \lambda_k \langle F(x_k), x_k - v_{k+1}\rangle - \tfrac{1}{2}\|v_k - v_{k+1}\|^2}_{\textbf{III}}. 
\end{equation}
Recall that $x_k = v_k - \gamma\|F(v_k)\|^{1/p-1}F(v_k)$, we have 
\begin{equation}\label{inequality:FO-descent-second}
x_k - v_k = - \gamma\|F(v_k)\|^{1/p-1}F(v_k).
\end{equation}
We perform a decomposition of $\|F(x_k) - F(v_k)\|$ and derive from Eq.~\eqref{inequality:FO-descent-second} that 
\begin{eqnarray*}
\lefteqn{\|F(x_k) - F(v_k)\| \leq \|F(x_k) - F(v_k) - \sum_{m=1}^{p-1} \tfrac{1}{m!}D^{(m)} F(v_k)[x_k-v_k]^m\| + \sum_{m=1}^{p-1} \tfrac{1}{m!}|D^{(m)} F(v_k)[x_k-v_k]^m|} \\
& = & \|F(x_k) - F(v_k) - \sum_{m=1}^{p-1} \tfrac{1}{m!}D^{(m)} F(v_k)[x_k-v_k]^m\| + \sum_{m=1}^{p-1} \tfrac{\gamma^m}{m!}\|F(v_k)\|^{\frac{m(1-p)}{p}}|D^{(m)} F(v_k)[F(v_k)]^m|.
\end{eqnarray*}
Since $F$ is $p^{\textnormal{th}}$-order strongly Lipschitz continuous, we have 
\begin{equation*}
\|F(x_k) - F(v_k) - \sum_{m=1}^{p-1} \tfrac{1}{m!}D^{(m)} F(v_k)[x_k-v_k]^m\| \leq \tfrac{L}{p!}\|x_k - v_k\|^p, 
\end{equation*}
and 
\begin{equation*}
\sum_{m=1}^{p-1} \tfrac{\gamma^m}{m!}\|F(v_k)\|^{\frac{m(1-p)}{p}}|D^{(m)} F(v_k)[F(v_k)]^m| \leq \left(\sum_{m=1}^{p-1} \tfrac{\gamma^m}{m!}\right)L\|F(v_k)\| = \left(\sum_{m=1}^{p-1} \tfrac{\gamma^{m-p}}{m!}\right)L\|x_k - v_k\|^p. 
\end{equation*}
Since $0 < \gamma < 1$, we have
\begin{equation}\label{inequality:FO-descent-third}
\|F(x_k) - F(v_k)\| \leq \left(\sum_{m=1}^p \tfrac{\gamma^{m-p}}{m!}\right)L\|x_k - v_k\|^p \leq \tfrac{1}{\gamma^{p-1}} \left(\sum_{m=1}^p \tfrac{1}{m!}\right)L\|x_k - v_k\|^p. 
\end{equation}
For simplicity, we let $c_p = \sum_{m=1}^p \tfrac{1}{m!}$. By using the argument as in Lemma~\ref{Lemma:HO-descent} with Eq.~\eqref{inequality:FO-descent-second} and~\eqref{inequality:FO-descent-third}, we have
\begin{eqnarray*}
\lefteqn{\langle F(x_k), x_k - v_{k+1}\rangle = \langle F(x_k) - F(v_k), x_k - v_{k+1}\rangle + \langle F(v_k), x_k - v_{k+1}\rangle} \\
& \leq & \langle F(x_k) - F(v_k), x_k - v_{k+1}\rangle - \tfrac{1}{\gamma^p}\|x_k - v_k\|^{p-1} \langle x_k - v_k, x_k - v_{k+1}\rangle \quad \textnormal{(cf. Eq.~\eqref{inequality:FO-descent-second})} \\
& \leq & \tfrac{L c_p}{\gamma^{p-1}}\|x_k - v_k\|^p \|x_k - v_{k+1}\| - \tfrac{1}{\gamma^p}\|x_k - v_k\|^{p-1} \langle x_k - v_k, x_k - v_{k+1}\rangle \quad \textnormal{(cf. Eq.~\eqref{inequality:FO-descent-third})} \\
& \leq & \tfrac{Lc_p}{\gamma^{p-1}}\|x_k - v_k\|^{p+1} + \tfrac{Lc_p}{\gamma^{p-1}}\|x_k - v_k\|^p\|v_k - v_{k+1}\| - \tfrac{1}{\gamma^p}\|x_k - v_k\|^{p-1} \langle x_k - v_k, x_k - v_{k+1}\rangle. 
\end{eqnarray*}
Since $\langle x_k - v_k, x_k - v_{k+1}\rangle \geq \|x_k - v_k\|^2 - \|x_k - v_k\|\|v_k - v_{k+1}\|$, we have
\begin{equation*}
\langle F(x_k), x_k - v_{k+1}\rangle \leq \tfrac{1+\gamma Lc_p}{\gamma^p}\|x_k - v_k\|^p\|v_k - v_{k+1}\| - \tfrac{1-\gamma Lc_p}{\gamma^p}\|x_k - v_k\|^{p+1}. 
\end{equation*}
Since $0 < \gamma < \frac{1}{2Lc_p}$, the interval $[\tfrac{\gamma^p}{6-6\gamma Lc_p}, \tfrac{\gamma^p}{2+2\gamma Lc_p}]$ is nonempty. Thus, it would be reasonable to define the function $\eta_2(p, L, \gamma)$ satisfying that 
\begin{equation*}
\tfrac{\gamma^p}{6-6\gamma Lc_p} \leq \eta_2(p, L, \gamma) \leq \tfrac{\gamma^p}{2+2\gamma Lc_p}. 
\end{equation*}
Since $\lambda_k = \frac{\eta_2(p, L, \gamma)}{\|x_k - v_k\|^{p-1}}$, we have $\frac{1}{6-6\gamma Lc_p} \leq \frac{\lambda_k\|x_k - v_k\|^{p-1}}{\gamma^p} \leq \frac{1}{2+2\gamma Lc_p}$. This yields
\begin{eqnarray}\label{inequality:FO-descent-fourth}
\textbf{III} & \leq & \sum_{k=1}^T \left(\tfrac{1+\gamma Lc_p}{\gamma^p}\lambda_k\|x_k - v_k\|^p\|v_k - v_{k+1}\| - \tfrac{1}{2}\|v_k - v_{k+1}\|^2 - \tfrac{1-\gamma Lc_p}{\gamma^p}\lambda_k\|x_k - v_k\|^{p+1}\right) \nonumber \\
& \leq & \sum_{k=1}^T \left(\tfrac{1}{2}\|x_k - v_k\|\|\|v_k - v_{k+1}\| - \tfrac{1}{2}\|v_k - v_{k+1}\|^2 - \tfrac{1}{6}\|x_k - v_k\|^2\right) \nonumber \\ 
& \leq & \sum_{k=1}^T \left(\max_{a \geq 0}\left\{\tfrac{1}{2}\|x_k - v_k\|a - \tfrac{1}{2}a^2\right\} - \tfrac{1}{6}\|x_k - v_k\|^2\right) \nonumber \\
& = & -  \tfrac{1}{24} \left(\sum_{k=1}^T \|x_k - v_k\|^2 \right). 
\end{eqnarray}
Plugging Eq.~\eqref{inequality:FO-descent-fourth} into Eq.~\eqref{inequality:FO-descent-first} yields that 
\begin{equation*}
\sum_{k=1}^T \lambda_k \langle F(x_k), x_k - x\rangle \leq E_0 - E_T + \langle s_T, x - x_0\rangle -\tfrac{1}{24}\left(\sum_{k=1}^T \|x_k - v_k\|^2 \right). 
\end{equation*}
This completes the proof. 
\end{proof}
\paragraph{Proof of Theorem~\ref{Thm:Global-Rate-First}:} Since Lemma~\ref{Lemma:FO-descent} guarantees that Lemma~\ref{Lemma:HO-error} still holds true for Algorithm~\ref{Algorithm:FO}, we have
\begin{equation}\label{inequality:FO-main-first}
\inf_{1 \leq k \leq T} \|x_k - v_k\|^2 \leq \tfrac{12\|x^\star - x_0\|^2}{T}.  
\end{equation}
Recall that Eq.~\eqref{inequality:FO-descent-third} in the proof of Lemma~\ref{Lemma:FO-descent} yields
\begin{equation*}
\|F(x_k) - F(v_k)\| \leq \tfrac{1}{\gamma^{p-1}} \left(\sum_{m=1}^p \tfrac{1}{m!}\right)L\|x_k - v_k\|^p. 
\end{equation*}
Since $0 < \gamma < \frac{1}{2Lc_p}$ where $c_p = \sum_{m=1}^p \tfrac{1}{m!}$, we have 
\begin{equation*}
\|F(x_k) - F(v_k)\| \leq \tfrac{1}{2\gamma^p}\|x_k - v_k\|^p. 
\end{equation*}
Recall that $x_k = v_k - \gamma\|F(v_k)\|^{1/p-1}F(v_k)$ yields $\|F(v_k)\| = \frac{1}{\gamma^p}\|x_k - v_k\|^p$, we have
\begin{equation*}
\|F(x_k)\| \leq \tfrac{1}{\gamma^p}\|x_k - v_k\|^p + \tfrac{1}{2\gamma^p}\|x_k - v_k\|^p \leq \tfrac{2}{\gamma^p}\|x_k - v_k\|^p.  
\end{equation*}
By the definition of $\textsc{res}(\cdot)$, we have
\begin{equation}\label{inequality:FO-main-second}
\inf_{1 \leq k \leq T} \textsc{res}(x_k) = \inf_{1 \leq k \leq T} \|F(x_k)\| \leq \tfrac{2}{\gamma^p}\left(\inf_{1 \leq k \leq T}\|x_k - v_k\|^p\right). 
\end{equation}
Plugging Eq.~\eqref{inequality:FO-main-first} into Eq.~\eqref{inequality:FO-main-second} yields the desired result. 

\subsection{Proof of Theorem~\ref{Thm:Local-Rate-Estimate}}
We derive from Eq.~\eqref{inequality:HO-main-first} and Eq.~\eqref{inequality:HO-main-second} in the proof of Theorem~\ref{Thm:Global-Rate-High} that one iteration of Algorithm~\ref{Algorithm:HO} with an input $x_0$ satisfies that 
\begin{equation*}
\|F(x_1)\| \leq \tfrac{(2p+1)L}{p!}(12\|x_0 - x^\star\|^2)^{\frac{p}{2}} \leq \tfrac{4^p(2p+1)L}{p!}\|x_0 - x^\star\|^p. 
\end{equation*}
By abuse of notation, we let the iterates $\{x_k\}_{k \geq 0}$ be generated by Algorithm~\ref{Algorithm:Restart} with Algorithm~\ref{Algorithm:HO} as the subroutine. Then, $x_{k+1}$ is an output of Algorithm~\ref{Algorithm:HO} with the input $(p, x_k, L, 1)$, i.e., one iteration of Algorithm~\ref{Algorithm:HO} with an input $x_k$.  This implies that 
\begin{equation*}
\|F(x_{k+1})\| \leq \tfrac{4^p(2p+1)L}{p!}\|x_k - x^\star\|^p. 
\end{equation*}
Since $F$ is $\mu$-strongly monotone, we have
\begin{equation*}
\|x_{k+1} - x^\star\|^2 \leq \tfrac{1}{\mu}\langle F(x_{k+1}), x_{k+1} - x^\star\rangle \leq \tfrac{1}{\mu}\|F(x_{k+1})\|\|x_{k+1} - x^\star\|. 
\end{equation*}
Putting these pieces together yields that 
\begin{equation*}
\|x_{k+1} - x^\star\| \leq \tfrac{1}{\mu}\|F(x_{k+1})\| \leq \tfrac{4^p(2p+1)}{p!}\tfrac{L}{\mu}\|x_k - x^\star\|^p, 
\end{equation*}
which implies that 
\begin{equation}
\|x_{k+1} - x^\star\| \leq \left(\tfrac{4^p(2p+1)}{p!}\tfrac{L}{\mu}\right)\|x_k - x^\star\|^p. 
\end{equation}
For the case of $p \geq 2$, we have $p-1 \geq 1$. If the following condition holds true: 
\begin{equation*}
\|x_0 - x^\star\| \leq \tfrac{1}{2}\left(\tfrac{p!}{4^p(2p+1)}\tfrac{\mu}{L}\right)^{\frac{1}{p-1}},  
\end{equation*}
then we have
\begin{eqnarray*}
\lefteqn{\left(\tfrac{4^p(2p+1)}{p!}\tfrac{L}{\mu}\right)^{\frac{1}{p-1}}\|x_{k+1} - x^\star\| \leq \left(\tfrac{4^p(2p+1)}{p!}\tfrac{L}{\mu}\right)^{\frac{p}{p-1}}\|x_k - x^\star\|^p} \\
& = & \left(\left(\tfrac{4^p(2p+1)}{p!}\tfrac{L}{\mu}\right)^{\frac{1}{p-1}}\|x_k - x^\star\|\right)^p \leq \left(\left(\tfrac{4^p(2p+1)}{p!}\tfrac{L}{\mu}\right)^{\frac{1}{p-1}}\|x_0 - x^\star\|\right)^{p^{k+1}} \leq (\tfrac{1}{2})^{p^{k+1}}. 
\end{eqnarray*}
For Algorithm~\ref{Algorithm:FO}, we have
\begin{equation*}
\|F(x_1)\| \leq \tfrac{2}{\gamma^p}(12\|x_0 - x^\star\|^2)^{\frac{p}{2}} \leq \tfrac{2^{2p+1}}{\gamma^p}\|x_0 - x^\star\|^p. 
\end{equation*}
By a similar argument, we have
\begin{equation*}
\|x_{k+1} - x^\star\| \leq \left(\tfrac{2^{2p+1}}{\mu\gamma^p}\right)\|x_k - x^\star\|^p. 
\end{equation*}
Thus, if the following condition holds true: 
\begin{equation*}
\|x_0 - x^\star\| \leq \tfrac{1}{2}\left(\tfrac{\mu\gamma^p}{2^{2p+1}}\right)^{\frac{1}{p-1}}, 
\end{equation*}
we have
\begin{equation*}
\left(\tfrac{2^{2p+1}}{\mu\gamma^p}\right)^{\frac{1}{p-1}}\|x_{k+1} - x^\star\| \leq (\tfrac{1}{2})^{p^{k+1}}. 
\end{equation*}
This completes the proof. 

%% file: sec/exp.tex
%!TEX root = paper.tex
\section{Experiments}\label{sec:exp}
We evaluate the performance of Algorithm~\ref{Algorithm:FO} for solving unconstrained min-max optimization problems with synthetic datasets. The baseline method is the first-order extragradient (EG) method.\footnote{We do not evaluate Algorithm~\ref{Algorithm:HO} and other related high-order methods since their effectiveness have been demonstrated by the experiments in prior works~\citep{Bullins-2022-Higher, Lin-2022-Explicit}. } Both our method and the EG method were implemented using MATLAB R2021b on a MacBook Pro with an Intel Core i9 2.4GHz and 16GB memory.
\begin{figure*}[!t]
\centering
\includegraphics[width=0.48\textwidth]{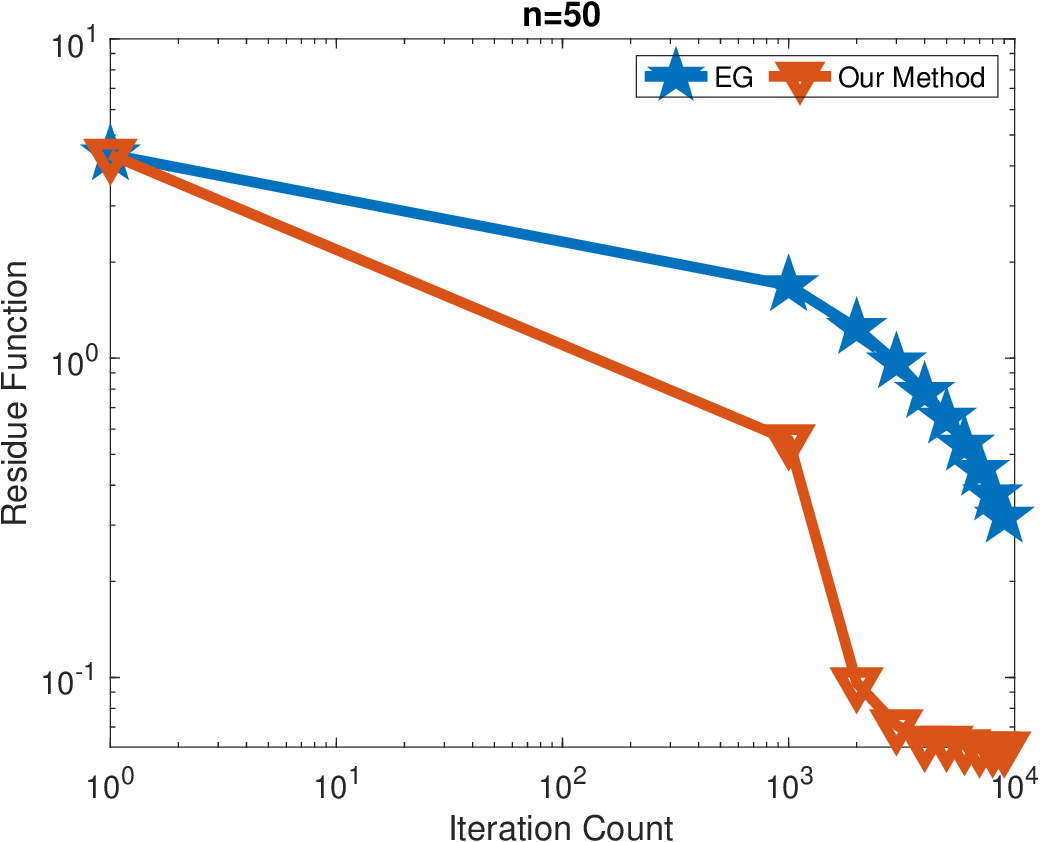}
\includegraphics[width=0.48\textwidth]{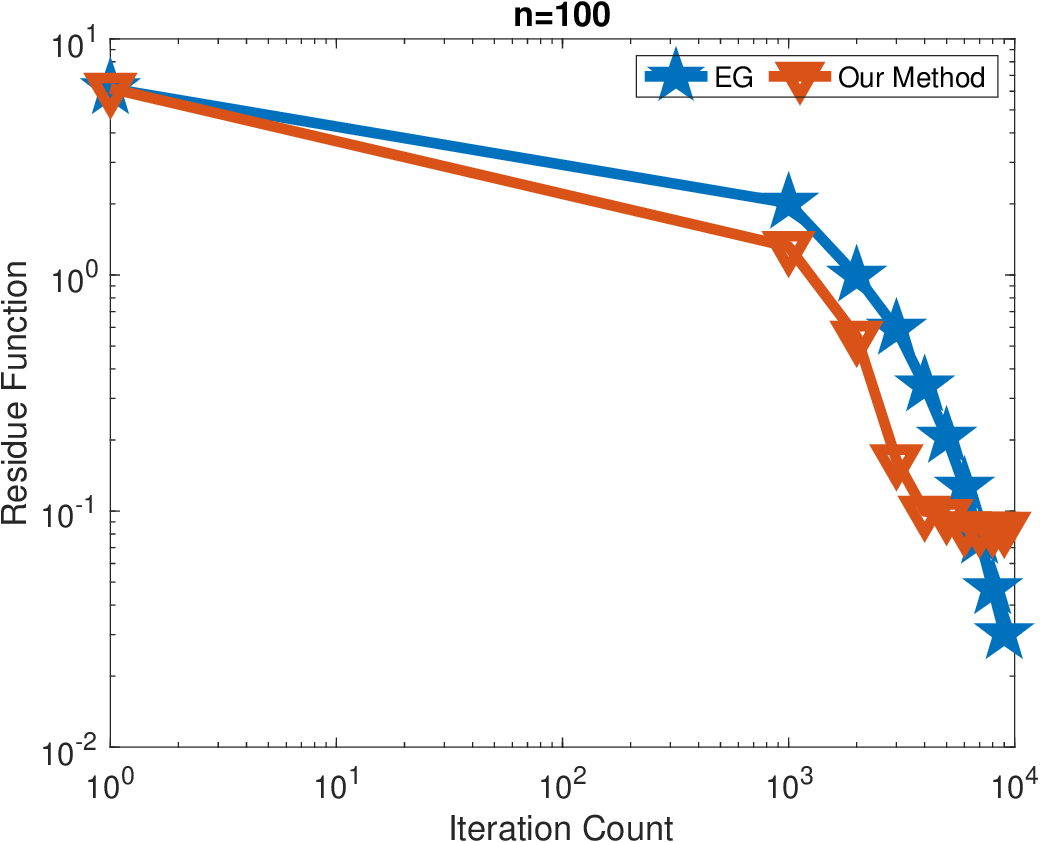} \\ \vspace{5pt}
\includegraphics[width=0.48\textwidth]{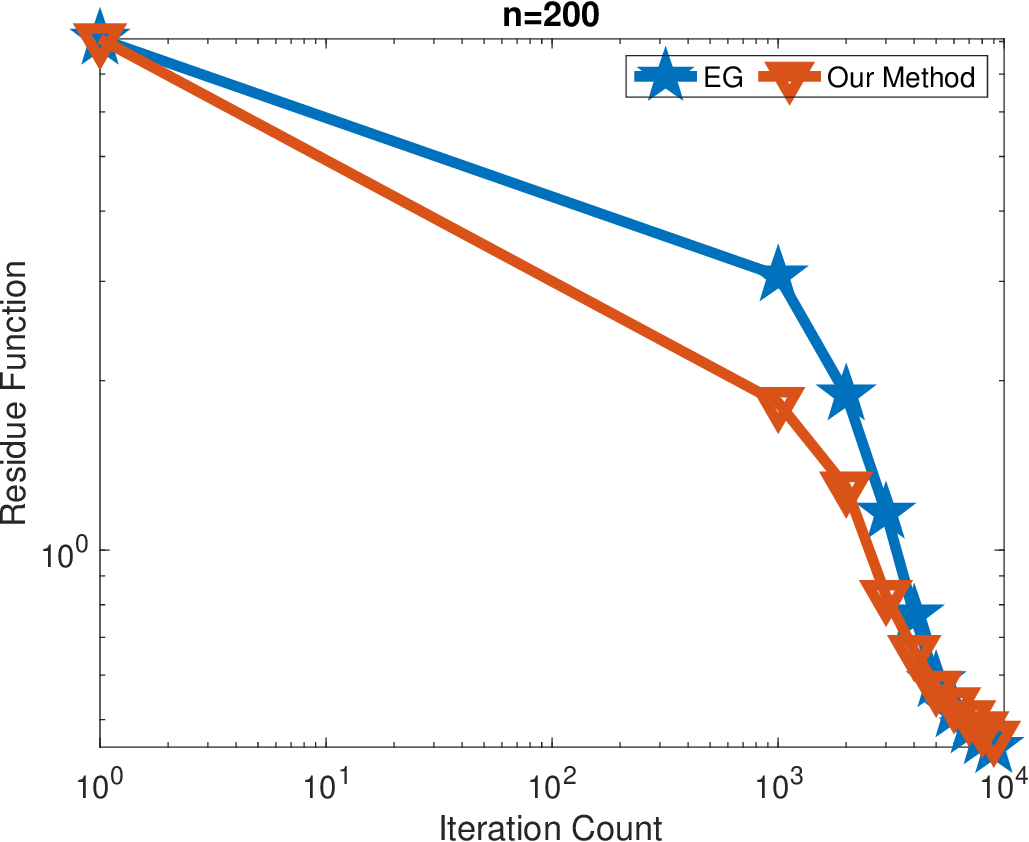}
\includegraphics[width=0.48\textwidth]{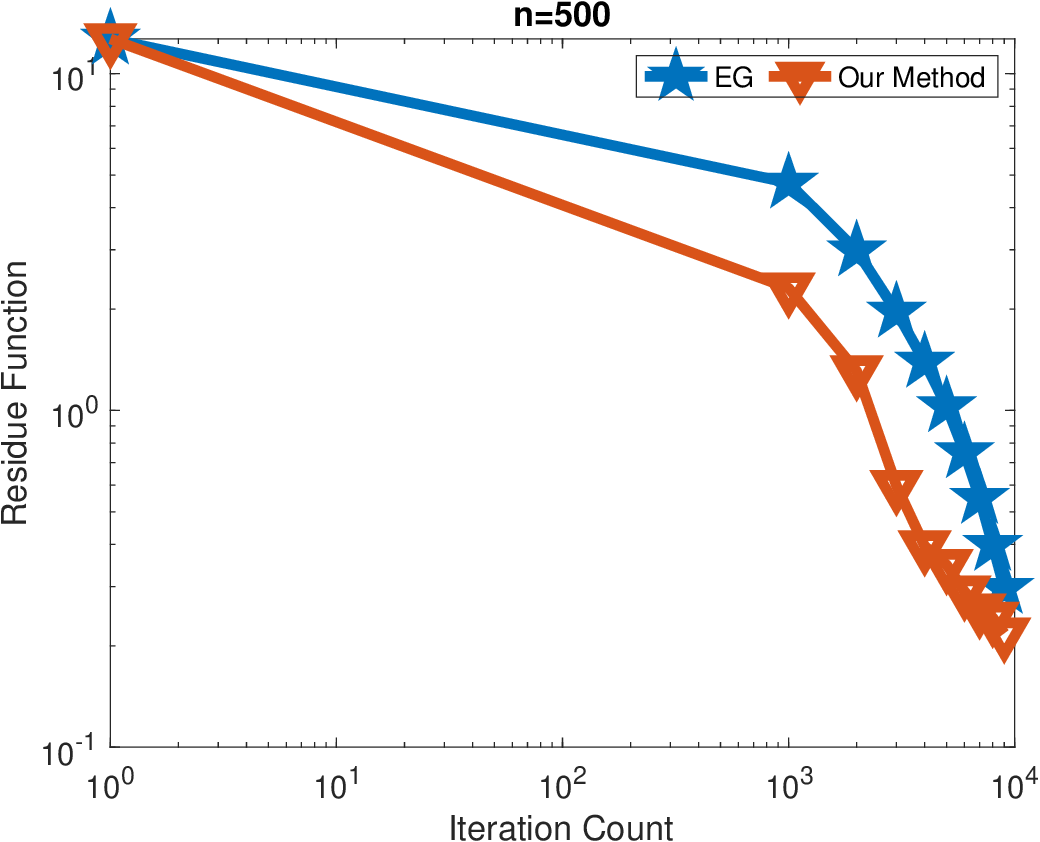}
\vspace*{-.5em}\caption{Performance of all the algorithms for $n \in \{50, 100, 200\}$ when $\rho = \frac{1}{20n}$ is set. The numerical results are presented in terms of iteration count (Top) and computational time (Bottom).}\label{fig:results}\vspace*{-1em}
\end{figure*}

Following the setup of~\cite{Jiang-2022-Generalized}, we consider a problem in the following form: 
\begin{equation}\label{prob:quartic}
\min_{z \in \br^n} \max_{y \in \br^n} \ f(z, y) = \tfrac{\rho}{24}\|z\|^4 + y^\top(Az - b), 
\end{equation}
where $\rho > 0$, the entries of $b \in \br^n$ are generated independently from $[-1, 1]$ and $A \in \br^{n \times n}$ is given by
\begin{equation*}
A = \begin{bmatrix}
1 & -1 & & & \\ & 1 & -1 & & \\ & & \ddots & \ddots & \\ & & & 1 & -1 \\ & & & & 1
\end{bmatrix}. 
\end{equation*}
This min-max optimization problem is convex-concave and has a unique global solution $z^\star = A^{-1}b$ and $y^\star = -\frac{\rho}{6}\|z^\star\|^2(A^\top)^{-1}z^\star$. It can be reformulated in the form of ME problem (cf. Eq.~\eqref{prob:main}) with 
\begin{equation*}
x = \begin{bmatrix} z \\ y \end{bmatrix}, \qquad F(x) = \begin{bmatrix} \tfrac{\rho}{6}\|z\|^2 z + A^\top y \\ -(Az-b) \end{bmatrix}. 
\end{equation*}
In our experiment, we set $\rho = \frac{1}{100n}$ where $n \in \{50, 100, 200, 500\}$ varies and use the residue function as the evaluation metric. We set the stepsize in the EG method as $0.05$ and $p = 3$ in Algorithm~\ref{Algorithm:FO}. Our results are summarized in Figure~\ref{fig:results} and we can see that our new method outperforms the EG method in terms of solution accuracy. This is because that our method can exploit the special structure of Eq.~\eqref{prob:quartic}, demonstrating the potential to design new gradient-based algorithms for solving structured application problems. Our empirical findings coincide with a line of existing works~\citep{Maddison-2018-Hamiltonian, Zhang-2018-Direct, Wilson-2019-Accelerating, Donoghue-2019-Hamiltonian, Loizou-2020-Stochastic}, which designed new accelerated gradient-based methods for solving structured problems (e.g., with strongly smooth loss functions) with provably fast global convergence rate and favorable numerical results.

%% file: sec/conclu.tex
%!TEX root = paper.tex
\section{Conclusions}\label{sec:conclusions}
We have presented a new class of accelerated rescaled gradient systems that yield global acceleration in the general setting of  monotone equation (ME) problems. Our analysis shows that our systems are equivalent to closed-loop control systems~\citep{Lin-2023-Monotone}, which allows for establishing desired properties of solution trajectories, including global existence and uniqueness, global asymptotic weak and/or strong convergence,  and global convergence rate estimation in terms of a residue function.  Our framework provides a systematic approach to deriving existing high-order methods and additionally to derive a new suite of simple first-order methods for solving ME problems. The same global rate in terms of a residue function is established for these methods under suitable high-order Lipschitz continuity conditions. For future research, it would be of interest to bring our continuous-time perspective for understanding various ME methods into register with the Lagrangian and Hamiltonian frameworks that have proved productive~\citep{Wibisono-2016-Variational, Diakonikolas-2021-Generalized, Muehlebach-2021-Optimization}.